\documentclass[10pt,a4paper]{amsart}
\usepackage{amsfonts}
\usepackage{amsthm}
\usepackage{amsmath}
\usepackage{amscd}
\usepackage[latin2]{inputenc}
\usepackage{t1enc}
\usepackage[mathscr]{eucal}
\usepackage{indentfirst}
\usepackage{graphicx}
\usepackage{graphics}
\usepackage{pict2e}
\usepackage{epic}
\usepackage[normalem]{ulem}
\numberwithin{equation}{section}
\usepackage{epstopdf}
\usepackage{verbatim}
\usepackage{amssymb}
\usepackage{tikz-cd}

\usepackage[leqno]{amsmath}
\makeatletter

\usepackage{hyperref}
\AtBeginDocument{}

\theoremstyle{plain}
\newtheorem{Th}{Theorem}[section]
\newtheorem{Lemma}[Th]{Lemma}
\newtheorem{Cor}[Th]{Corollary}
\newtheorem{Prop}[Th]{Proposition}

\theoremstyle{definition}
\newtheorem{Def}[Th]{Definition}

\newtheorem{Rem}[Th]{Remark}
\newtheorem{?}[Th]{Problem}

\def\al{\alpha}
\def\be{\beta}

\def\w{\wedge}
\def\R{\mathbb{R}}
\def\C{\mathbb{C}}

\def\Lm{\Lambda}

\def\om{\omega}
\def\Om{\Omega}

\def\ip{\raise1pt\hbox{\large$\lrcorner$}\>}

\DeclareMathOperator{\vol}{vol}

\DeclareMathOperator{\Ric}{Ric}

\begin{document}
	
\title[Abelian instantons on $S^1$-invariant K\"ahler Einstein $6$-manifolds]{Explicit abelian instantons on $S^1$-invariant K\"ahler Einstein $6$-manifolds}

\author[U. Fowdar]{Udhav Fowdar}

\address{IMECC - Unicamp \\ 
	Rua S\'ergio Buarque de Holanda, 651\\
	13083-859\\
	Campinas - SP\\
	Brazil}

\email{udhav@unicamp.br}

\keywords{Hermitian Yang-Mills, Special Lagrangians, Circle action} 
\subjclass[2010]{MSC 53C07, 53C38, 53C55}

\begin{abstract} 
We consider a dimensional reduction of the (deformed) Hermitian Yang-Mills condition on $S^1$-invariant K\"ahler Einstein $6$-manifolds. This allows us to reformulate the (deformed) Hermitian Yang-Mills equations in terms of data on the quotient K\"ahler $4$-manifold. In particular, we apply this construction to 
the canonical bundle of $\mathbb{C}\mathbb{P}^2$ endowed with the Calabi ansatz metric
to find explicit abelian $SU(3)$ instantons and we show that these are determined by the spectrum of $\mathbb{C}\mathbb{P}^2$. We also find $1$-parameter families of explicit deformed Hermitian Yang-Mills connections. As a by-product of our investigation we find a coordinate expression for its holomorphic volume form 
which leads us to construct a special Lagrangian foliation of $\mathcal{O}_{\mathbb{C}\mathbb{P}^2}(-3)$.    
\end{abstract}

\maketitle
\tableofcontents

\section{Introduction}
Given a K\"ahler $6$-manifold $(P^6,g,J,\om)$ and a Hermitian vector bundle $E\to P^6$, a connection form $\underline{A}$ on $E$ is called Hermitian Yang-Mills if its curvature form $F_{\underline{A}}$ is of type $(1,1)$ with respect to $J$ and $g(F_{\underline{A}},\om)$ is a constant multiple of the identity. If $P^6$ admits an $S^1$ action preserving all the above data then one can reduce the problem of finding Hermitian Yang-Mills connections on $P^6$ to solving a twisted Bogomolny type system of PDEs on the K\"ahler quotient $M^4:=P^6/\!\!/S^1$. This paper is concerned with the problem of finding explicitly such $S^1$-invariant connections when the gauge group is $U(1)$ and $P^6$ is a \textit{non-compact} K\"ahler Einstein manifold, in particular we focus on the Calabi-Yau case. The main result of this paper is that (a subset of) the spectrum of $M^4$ determines such $SU(3)$ instantons and hence the moduli space of these instantons is identified with (a subset of) the eigenfunctions of the Laplacian of $M^4$. A detailed study in made in the case when $P^6=\C^3$ endowed with its standard Euclidean structure and $P^6=\mathcal{O}_{\C\mathbb{P}^2}(-3)$, the canonical bundle of $M^4=\C\mathbb{P}^2$, endowed with the Calabi ansatz metric. In the latter cases our result can be stated as follows: let $\mu_k=4k(k+2)$ denote an eigenvalue of the Laplacian of $\C\mathbb{P}^2$ then
\begin{Th}
	For each positive integer $k$ there exist non-trivial $S^1$-invariant abelian $SU(3)$-instantons on $\C^3$.
	If $k$ is of the form $3n$ or $3n+1$ for $n\in \mathbb{Z}^+_0$, then there exist non-trivial  $S^1$-invariant  abelian $SU(3)$-instantons on $\mathcal{O}_{\C\mathbb{P}^2}(-3)$.
\end{Th}
For a more precise statement see Theorem \ref{maintheoremcanonicalbundle}.
Alongside Hermitian Yang-Mills connections, we also consider the deformed Hermitian Yang-Mills equation (see definition \ref{dhymdefinition}). In particular, in this case we show that
\begin{Th}
There exist three $1$-parameter families of $S^1$-invariant deformed Hermitian Yang-Mills connections on $\mathcal{O}_{\C\mathbb{P}^2}(-3)$. Moreover, exactly one connection in these families is also a (non-trivial) Hermitian Yang-Mills connection.
\end{Th}
For a more precise statement see Theorems \ref{dhymtheorem} and \ref{dhymC3theorem}.
Our investigation also leads us a coordinate expression for the holomorphic volume form $\Om$ on $\mathcal{O}_{\C\mathbb{P}^2}(-3)$ and using this we show that
\begin{Th}
Outside a set of complex codimension $1$, $\mathcal{O}_{\C\mathbb{P}^2}(-3)$ admits a foliation by special Lagrangians. As the zero section $\C\mathbb{P}^2$ shrinks to a point so that $\mathcal{O}_{\C\mathbb{P}^2}(-3) \to \C^3$ this descends to a special Lagrangian foliation of $\C^3$ by flat $\R^+_0 \times \R^2$.
\end{Th}
For a more precise statement see Theorem \ref{speciallagfoliation}.
We also construct abelian instantons on the canonical bundle of $S^2 \times S^2$ and on certain negative K\"ahler Einstein manifolds. For positive K\"ahler Einstein manifolds we only find instantons which have isolated conical singularities. Our results illustrate how constructing instantons depends on the K\"ahler Einstein structure of $P^6$ as well as that of $M^4$. 

\textbf{Motivation.}
The Donaldson-Uhlenbeck-Yau theorem asserts that for \textit{compact} K\"ahler manifolds, there is a correspondence between irreducible Hermitian Yang-Mills connections on the vector bundle $E$ and stable holomorphic structures on $E$ cf. \cite{Donaldson1987,UhlenbeckYau1986}. Thus, this allows one to translate the differential geometric problem of constructing Hermitian Yang-Mills connections to an algebraic geometry problem. This correspondence is however not known in general for non-compact manifolds, unless the asymptotic geometry is strongly constrained cf. \cite{Marcos2017, SaEarp2015}. 
So in the non-compact setting a more fruitful approach has been to exploit the symmetry of the underlying space to simplify the instanton condition. 
In the presence of a cohomogeneity one action, in \cite{Oliveira2016} Oliveira studies instantons on the cotangent bundle of $S^3$ endowed with the Stenzel metric \cite{Stenzel1993}. Similar ideas were extended to construct $G_2$-instantons on the spinor bundle of $S^3$ endowed with the Bryant-Salamon \cite{Bryant1989} and Brandhuber et al. \cite{Brandhuber01} metrics in \cite{Lotay2016}, and more recently these techniques were used in \cite{Turner2022, Papoulias2021, Stein2022} to construct $G_2,Spin(7)$ and $SU(3)$ instantons respectively. Motivated by the latter works we considered the problem of constructing $SU(3)$ instantons on the canonical bundle of $\C\mathbb{P}^2$, which admits a cohomogeneity one action of $SU(3)$. Unfortunately due the scarcity of invariant connections this approach does not yield any new instantons. Thus, this led us to consider a small symmetry group namely just a $U(1)$ action. The advantage of considering a smaller symmetry group is that we can tackle the problem of constructing instantons on a much larger class of manifolds but the drawback is that the reduced equations are more complicated, so we were only able to find new solutions in the abelian case. This is nonetheless especially interesting in the case when $P^6$ is a Calabi-Yau $3$-fold in view of the SYZ conjecture \cite{SYZ1996}, which suggests that abelian instantons restricting to flat connections on special Lagrangians play an important role in mirror symmetry. Along with Hermitian Yang-Mills connections, we also consider the deformed Hermitian Yang-Mills equation introduced in \cite{Leung2000, Marino2000} which are in some sense ``mirror'' to special Lagrangians.
One key goal of this article is to derive as explicit expressions as possible which can hopefully find applications in other works. We now elaborate on the results in this paper.

\textbf{Outline of paper.}
In Section \ref{preliminaries} we recall the basic definitions and fix some conventions.
The general set up of our paper is described in Section \ref{kahlerreductionsection}.
We begin by showing how $S^1$-invariant K\"ahler structures on $P^6$ can be constructed starting from suitable data on the quotient space $M^4$, see Theorem \ref{kahlerreductiontheorem}. We then specialise this to the case when such data leads to K\"ahler Einstein metrics, see Corollary \ref{EinsteinConditionCorollary}. By solving the resulting equations under certain assumptions, we describe explicitly the K\"ahler Einstein structures on a large class of both complete and incomplete examples, see cases (\ref{calabiansatzgenral})-(\ref{EinsteinonHK}), sub-sections \ref{pqnonzeroexamples} and \ref{nonconstantsolutions}. 

In Section \ref{kahlerreductionofHYMsection} we then consider the $S^1$ reduction of the Hermitian Yang-Mills equations to $M^4$, see Theorem \ref{instantonkahlerreductiontheorem}. The resulting system of equations involves $1$-parameter families of Higgs fields $\psi$ and Hermitian connections $A$ on $M^4$. We specialise to the aforementioned K\"ahler Einstein manifolds, see Corollary \ref{S1invariantInstantonconstantU}, and find that the Higgs field $\psi$ satisfies a, generally non-linear, second order elliptic PDE. When the gauge group is abelian, the latter equation becomes linear and with a suitable ansatz we show that such solutions $\psi$ are determined by an ODE depending on the spectrum of $M^4$. The delicate problem becomes to understand whether these local solutions give rise to globally well-defined instantons on $P^6$. While in general such problems require ODE analysis and methods of dynamical systems (since the ODEs can have singularities depending on the topology of $P^6$ cf. \cite{Lotay2016, Turner2022}), we were surprisingly able to find explicit solutions in several cases.  We also consider the corresponding $S^1$ reduction of the deformed Hermitian Yang-Mills equation on the K\"ahler Einstein manifolds, see Theorem \ref{S1invariantdhymtheorem}. With a suitable ansatz we also find explicit solutions to the latter in Section \ref{dhymsection}.

In Section \ref{CP2section} we consider the case when $P^6=\mathcal{O}_{\C\mathbb{P}^2}(-3)$ is the canonical bundle of $M^4=\C\mathbb{P}^2$. Since $\mathcal{O}_{\C\mathbb{P}^2}(-3)$ can be thought as a deformation of the $\C^3$ (just as the Eguchi-Hanson space is for $\C^2$) we also consider instantons on $\C^3$.
We show that while only a subset of the spectrum of $\C\mathbb{P}^2$ yields well-defined instantons on $\mathcal{O}_{\C\mathbb{P}^2}(-3)$, all positive eigenvalues of $\C\mathbb{P}^2$ yield instantons on $\C^3$, see Theorem \ref{maintheoremcanonicalbundle}. We also observe that only one of our instantons, corresponding to the zero eigenvalue, has finite Yang-Mills energy (depending on the size of the zero section $\C\mathbb{P}^2$).
In the process of describing the $SU(3)$-structure on $\mathcal{O}_{\C\mathbb{P}^2}(-3)$, we derive a coordinate expression for the holomorphic volume form, see Proposition \ref{holomorphicvolumeformCP2}. Using this we construct a special Lagrangian foliation of $\mathcal{O}_{\C\mathbb{P}^2}(-3)$ which also descends to a foliation of $\C^3$ by flat $\R^+_0 \times \R^2$, see Theorem \ref{speciallagfoliation}. 

In Section \ref{S2S2instantonsection} we consider the case when $P^6$ is the canonical bundle of $S^2 \times S^2$ and highlight the similarities and differences with the results of Section \ref{CP2section}. We also derive a coordinate expression for the holomorphic volume form in this case (see Proposition \ref{holomorphicvolumeformS2S2}) but since the expression is more complicated we were only able to find $1$-parameter families of special Lagrangians. The latter descend to flat cones on $\mathbb{T}^2 \subset S^2 \times S^2$ in the associated Calabi-Yau cone as the zero section is collapsed to a point.

In Section \ref{T4section} we consider abelian instantons on the K\"ahler Einstein manifolds with positive and negative scalar curvature described in Section \ref{kahlerreductionsection}. Finally in Section \ref{dhymsection} we construct solutions to the deformed Hermitian Yang-Mills connections.

\noindent\textbf{Acknowledgements.}
The author would like to thank Marcos Jardim and Jason Lotay for stimulating conversations that helped shape this article. 
The author was supported by the S\~ao Paulo Research Foundation (FAPESP) [2021/07249-0].

\section{Preliminaries}\label{preliminaries}
In this section we define the basic objects of interest in this paper and fix some notations and conventions. We begin by recalling the notion of an $SU(3)$-structure.
\begin{Def}
	An $SU(3)$-structure on a $6$-manifold $P^6$ is given by a non-degenerate $2$-form $\om$, a Riemannian metric $g_\om$, an almost complex structure $J$ and a complex $3$-form $\Om:=\Om^++i\Om^-$ satisfying 
	\begin{gather} 
		\om \w \Om^\pm=0, \label{compatibilitysu3}\\
		\frac{2}{3} \om^3={\Om^+ \w \Om^-} .\label{normalisationsu3}
	\end{gather}
\end{Def}
Although an $SU(3)$-structure consists of the data $(g_\om,\om,J,\Om)$ it is in fact sufficient to specify the pair $(\om,\Om^+)$ satisfying (\ref{compatibilitysu3}) and (\ref{normalisationsu3}). This follows from an observation of Hitchin that the stabiliser of $\Om^+$ in ${GL}^+(6,\R)$ is congruent to ${SL}(3,\C)\subset {GL}(3,\C)$ and hence $\Om^+$ determines $J$ cf. \cite{Hitchin2000}. The metric is then determined by
\begin{equation}
	\om(\cdot,\cdot):=g_{\om}(J\cdot,\cdot),\label{omJandg}
\end{equation}
and $\Om^-:=J(\Om^+)=*_\om\Om^+$, where $*_\om$ is the Hodge star operator determined by $g_\om$ and volume form:
\[\vol_\om:=\frac{1}{6}\om^3=\frac{1}{4}\Om^+\w\Om^-.\]
Since $J^2=-1$ we can decompose the complexified space of $1$-forms into $+i$ and $-i$ $J$-eigenspaces which we denote by $\Lm^{1,0}$ and $\Lm^{0,1}$ respectively.
We denote by $\Lm^{p,q}$ the space of differential $(p+q)$-forms obtained by wedging $p$ elements of  $\Lm^{1,0}$ and $q$ element of $\Lm^{0,1}$. For a real function $f$ we define $d^cf(\cdot):=J \circ df(\cdot)=df (J(\cdot))$, so that $df-id^cf\in\Lm^{1,0}$. We should point out that our definition of the operator $d^c$ might differ from other conventions in the literature by a minus sign. We define the inner product on decomposable $k$-forms by
\[g(\al_1\w \cdots \w \al_k,\be_1\w \cdots \w \be_k)=\det(g(\al_i,\be_j)),\]
which might differ from other conventions in the literature by a `$k!$' factor and we denote the induced norm by $\|\cdot\|_{g}$.

We say $(P^6,\om,\Om^+)$ is a Calabi-Yau $3$-fold if both $\om$ and $\Om^+$ are covariantly constant with respect to the Levi-Civita connection $\nabla$, and hence by the above discussion so are $J$ and $\Om^-$. This condition can be equivalently formulated as:
\begin{Th}[\cite{ChiossiSalamonIntrinsicTorsion}]\label{su3torsionthm}
	$\nabla\om=0$ and $\nabla\Om^+=0$ if and only if $d\om=0$ and $d\Om^+=d\Om^-=0$.
\end{Th}
Note that by contrast to a Calabi-Yau manifold, a K\"ahler manifold only admits \textit{locally} a complex volume form $\Om$ and rather than being closed it satisfies the weaker condition that $d\Om\in \Lm^{3,1}.$ Henceforth $P^6$ will always be a K\"ahler manifold, and we shall emphasise when we further require $P^6$ to be Calabi-Yau manifold. Next we recall the notion of a Hermitian Yang-Mills connection.
\begin{Def}
	Given a Hermitian vector bundle $E$ over a K\"ahler manifold $P^6$ with a connection $\underline{A}$ (viewed as a matrix of $1$-form), we say that $\underline{A}$ is a Hermitian Yang-Mills connection it satisfies
	\begin{gather}
		F_{\underline{A}} \in \Lm^{1,1}(P^6),\label{hym11}\\
		F_{\underline{A}} \w \om^2 = -2C_0 \vol_\om,\label{hym2}
	\end{gather}
	where $F_{\underline{A}}:=d\underline{A} + \underline{A}\w \underline{A}$ denotes the curvature form of $\underline{A}$ and $C_0$ is a constant. When $C_0=0$ we shall call such connections $SU(3)$-instantons, or simply instantons, cf. \cite{Carrion1998}.
\end{Def}
If $P^6$ is a Calabi-Yau $3$-fold then (\ref{hym11}) can equivalently be expressed as
\begin{equation}
	F_{\underline{A}} \w \Om^+ =0.\label{hym1}
\end{equation}
It is worth pointing out that when $P^6$ is compact, the constant $C_0$ essentially corresponds to the slope of the vector bundle $E$. Next we now recall the notion of a deformed Hermitian Yang-Mills connection cf. \cite{Leung2000, Lotay2020}:
\begin{Def}\label{dhymdefinition}
Given a Hermitian \textit{line} bundle $L$ over a K\"ahler manifold $P^6$ with a connection $\underline{A}$, we say that $\underline{A}$ is a deformed Hermitian Yang-Mills connection it satisfies (\ref{hym11}) and
\begin{equation}
	\frac{1}{6}F_{\underline{A}}^3 - F_{\underline{A}} \w \frac{1}{2}\om^2 =0\label{dhym}.
\end{equation}
\end{Def}
We should emphasise that here we are viewing $F_{\underline{A}}$ as a \textit{real} $2$-form by identifying $\mathfrak{u}(1)=i\R \cong \R$ which explains the minus sign in (\ref{dhym}) compared to \cite[Def. 3.1]{Lotay2020}. Observe also that (\ref{dhym}) is a non-linear equation but it does have a $\mathbb{Z}_2$-symmetry i.e. if $F_{\underline{A}}$ is a solution then so is $-F_{\underline{A}}$.
Finally we recall the definition of special Lagrangian submanifolds.

Calabi-Yau manifolds have a natural class of calibrated, and hence minimal, submanifolds called special Lagrangians defined as follows:
\begin{Def}
	A $3$-dimensional submanifold $L^3$ of $P^6$ is called a special Lagrangian if it is Lagrangian i.e. $\om\big|_L=0$ and furthermore $\Om^-\big|_L=0$, or equivalently $\Om^+$ calibrates $L^3$ i.e. $\Om^+\big|_L=\vol_L$.
\end{Def}
Aside from special Lagrangians there are also other examples of calibrated submanifolds, namely holomophic curves and surfaces which are instead calibrated by $\om$ and $\om^2/2$ respectively.
We refer the reader to \cite{HarveyLawson, Joycebook} for basic facts about general calibrated submanifolds and to \cite[Chap. 9]{Joycebook} for the (expected) role of special Lagrangians in mirror symmetry. Having now defined the geometrical structures of interest we next move on to describe the general set up of this paper. 
\section{$S^1$-invariant K{\"a}hler structures}\label{kahlerreductionsection}
The goal of this section is to describe, as explicitly as possible, all local $S^1$-invariant K\"ahler  metrics on $6$-manifolds (see Theorem \ref{kahlerreductiontheorem}) upon which we shall consider the Hermitian Yang-Mills equations. We shall also show how several well-known K\"ahler Einstein metrics fit into this framework (see sub-section \ref{constantsolutionsection}). 

\subsection{K\"ahler reduction of a Calabi-Yau $3$-fold}
Let $P^6$ be a Calabi-Yau $3$-fold whose $SU(3)$-structure is determined by the data $(g_\om,J,\om,\Om)$ and suppose furthermore that it admits a holomorphic Killing vector field $X$ i.e. 
\begin{equation*}
	\mathcal{L}_X\om=0, \ \ \mathcal{L}_XJ=0 \ \ \text{and}\ \ \mathcal{L}_Xg_\om=0.
\end{equation*}
Of course from (\ref{omJandg}) we know that any two of the above implies the third.
However we do \textit{not} assume that $X$ preserves $\Om$; but our hypothesis already implies that: 
\begin{Lemma}\label{permutingompluslemma}
	If $X$ is a holomorphic Killing vector field on $P^6$ then 
	\[\mathcal{L}_X\Om^\pm =\pm k\Om^\mp\]
	for some real constant $k$.	
\end{Lemma}
\begin{proof}
	Since $X$ preserves $J$, for any $(1,0)$ form $\alpha$  we have that
	\[J(\mathcal{L}_X\alpha)=\mathcal{L}_X(J\alpha)=i\mathcal{L}_X(\alpha)\]
	and hence it follows that $\mathcal{L}_X\alpha$ is also of type $(1,0)$. So
	in general 
	\begin{equation}
		\mathcal{L}_X\Om = -ki \Om
	\end{equation} 
	for some complex function $k:P^6\to \C$. Applying the exterior derivative to both sides we have that
	\[dk \w \Om =0\]
	and hence we see that $k$ is necessarily constant. Furthermore applying $\mathcal{L}_X$ to $g_\om(\Om^\pm,\Om^\pm)=4$ we deduce that $k\in\R$ and this concludes the proof. 
\end{proof}
\begin{Rem}
	Note that by rescaling the vector field $X$ we can always set $k=0$ or $1$ i.e. the $S^1$ action  generated by $X$ either preserves $\Om$ or corresponds to multiplication by $i$.
\end{Rem}
We shall now consider the K\"ahler reduction of $P^6$ by the $S^1$ action. Working away from the vanishing locus of $X$ i.e. the fixed point set of the $S^1$-action, we can define a positive function 
$$u:=g_\om(X,X)^{-1/2}$$
and a canonical Riemannian connection $1$-form by
$$\Theta:=u^2 g_\om(X,\cdot).$$ 
Note that by construction both $u$ and $\Theta$ are invariant by $X$. Locally on $P^6$ 
we can define a moment map $H$, i.e. a Hamiltonian function, by
\[dH=X\ip \om.\]
The data of the $SU(3)$-structure on $P^6$ can now be expressed as
\begin{gather}
	\om = \Theta \w dH + \tilde{\om}_1, \label{kahlerform}\\
	g_\om=u^{-2}\Theta^2+u^2dH^2+{g}_{\tilde{\om}_1},\label{kahlermetric}\\
	\Om = (u^{-1}\Theta+iudH) \w (\tilde{\om}_2+i\tilde{\om}_3),
\end{gather}
where $\tilde{\om}_i$ satisfy the $SU(2)$ relations: $\tilde{\om}_i \w \tilde{\om}_j = \delta_{ij}\tilde{\om}_1 \w \tilde{\om}_1$. We should emphasise that $u,\Theta,$ and $\tilde{\om}_i$ all depend on $H$. For regular values of $H$, we can define the K\"ahler quotient $M^4:=P^6\!/\!\!/\! S^1$ cf. \cite{HKLR}. The K\"ahler structure on $M^4$ is determined by the pair $(\tilde{g}_{\tilde{\om}_1},\tilde{\om}_1)$ for each fixed value of $H$, or put differently, we have a $1$-parameter family of K\"ahler structure on $M^4$ depending on $H$. Note however that induced complex structure on $M^4$ is always the same; it is the symplectic form, or equivalently by (\ref{omJandg}) the metric, which varies with $H$. 

We shall now show how to invert this reduction and recover the K\"ahler structure of $P^6$ (at least on an open set) starting from the data $(u,\tilde{\om}_1,\tilde{g}_{\tilde{\om}_1})$ on $M^4$.
\begin{Th}\label{kahlerreductiontheorem}
	Let $(M^4,J_1)$ be a complex $4$-manifold with a $1$-parameter family of K\"ahler form $\tilde{\om}_1(H)$ and a $1$-parameter family of positive function $u(H)$ satisfying 
	\begin{equation}
		\frac{\partial^2}{\partial H^2} (\tilde{\om}_1) = d_M d_M^c(u^2),\label{kahlerPDE}
	\end{equation}
	where $d_M$ denotes the exterior derivative on $M^4$ (i.e. we do not differentiate with respect to $H$) and $d_M^c:=J_1 \circ d_M$. Then locally we can define a connection $1$-form $\Theta$ on an $S^1$ bundle $Q^5 \to M^4$ whose curvature is given by
	\begin{equation}
		d\Theta=-\frac{\partial}{\partial H} \tilde{\om}_1+d_M^c(u^2)\w dH,\label{curvaturekahlerreduction}
	\end{equation}
	and $P^6:=\R_H \times Q^5$ admits a K\"ahler structure given by (\ref{kahlerform}) and (\ref{kahlermetric}).
	If furthermore $M^4$ admits a holomorphic $(2,0)$ form $\sigma_2+i\sigma_3$ such that 
	\begin{equation}
		u^2 (\sigma_2+i\sigma_3)\w (\sigma_2-i\sigma_3) = 2 (\tilde{\om}_1 \w \tilde{\om}_1)\label{uandomtilde}
	\end{equation}
	then we can define an $S^1$-invariant holomorphic $(3,0)$ form on $P^6$ by
	\begin{equation}
		\Om:=(\Theta + iu^2 dH) \w (\sigma_2+i\sigma_3)
	\end{equation}
	and hence $g_\om$ is a Calabi-Yau metric.
\end{Th}
\begin{proof}
	To prove the theorem it suffices to show that the quotient K\"ahler structure for $P^6 \to M^4$ indeed satisfies the hypothesis of the theorem. Since $\Theta$ is invariant under $X$, we can write $d\Theta=d_M \Theta -\partial_H(\Theta) \w dH $ (since $X\ip d\Theta=0$ there are no $\Theta$-terms); here $\partial_H(\Theta)$ denotes an $H$-dependent $1$-form on $M^4$. Since $d\om=0$ it follows that
	\begin{gather}
		d_M \Theta=-\partial_H(\tilde{\om}_1),\label{dmtheta}\\
		d_M\tilde{\om}_1=0.\label{dmom}
	\end{gather}
	The latter equation corresponds to the fact that $\tilde{\om}_1$ descends to a K\"ahler form on $M^4$ cf. \cite{HKLR}. Let us now define a $(2,0)$ form by $\sigma_2+i\sigma_3:=X\ip \Om$ then we have that
	\begin{align*}
		d(\sigma_2+i\sigma_3)&=d( X\ip \Om)\\ 
		&= \mathcal{L}_X\Om\\
		&=-ki\Om.
	\end{align*}
	It follows that
	\begin{align*}
		0 =d \Om 
		&= d(\Theta+iu^2dH )\w ( X\ip \Om)-(\Theta+iu^2dH )\w d( X\ip \Om)\\
		&=d(\Theta+iu^2dH )\w (\sigma_2+i\sigma_3)
	\end{align*}
	and from this we deduce that $d_M\Theta \in \Lm^{1,1}(M^4)$ and
	\begin{equation}
		\partial_H(\Theta)=-d^c_M(u^2). \label{dhtheta}
	\end{equation}
	Combining the latter and (\ref{dmtheta}) gives (\ref{kahlerPDE}) and this concludes the proof of the first part of the theorem. If we now set $k=0$ then $\sigma_2+i\sigma_3$ descends to a closed and hence \textit{holomorphic} $(2,0)$ form on $M^4$ and the second part of the theorem follows immediately from this and using (\ref{normalisationsu3}).
\end{proof}
\begin{Rem}\
\begin{enumerate}
\item Note that all $S^1$ invariant K\"ahler metrics on $6$-manifolds arise from the first part of Theorem \ref{kahlerreductiontheorem} i.e. we did not require $P^6$ to be a Calabi-Yau manifold in the proof of the theorem. The conditions (\ref{dmtheta}) and (\ref{dmom}) arise from the requirement that $d\om=0$ and condition (\ref{dhtheta}) arises from the requirement 
that the almost complex structure on the fibre is indeed integrable i.e. $d(\Theta +iu^2dH) \in \Lm^{2,0}\oplus\Lm^{1,1}.$ 
\item It is worth emphasising that Theorem \ref{kahlerreductiontheorem} is \textit{not} asserting that $g_\om$ is a Calabi-Yau metric if and only if $M^4$ admits a holomorphic $(2,0)$ form. It is merely saying that that $g_\om$ is a Calabi-Yau metric AND $P^6$ admits an $S^1$-invariant holomorphic volume form if and only if $M^4$ admits a holomorphic $(2,0)$ form $\sigma_2+i\sigma_3$, and as we saw above this only occurs if $k=0$. We shall see examples below when $g_\om$ is indeed a Calabi-Yau metric but $k\neq 0$. Irrespectively however, if $P^6$ is a Calabi-Yau manifold (not just K\"ahler) then we always have that $\sigma_2+i\sigma_3$ defines a transverse holomorphic $(2,0)$ form on $P^6\to M^4$ such that $\mathcal{L}_X(\sigma_2+i\sigma_3)=-ki(\sigma_2+i\sigma_3)$.
\end{enumerate}
\end{Rem}
Theorem \ref{kahlerreductiontheorem} only produces \textit{local} examples of K\"ahler metrics; the issue of completeness of $g_\om$ is more subtle and it is related to the zero locus of the vector field $X$ on $P^6$. Since in general we cannot characterise the Calabi-Yau condition on general $P^6$ from the above theorem, we shall next investigate the condition when the induced K\"ahler structure is Einstein.
\subsection{The Ricci form and the Einstein condition}
We now compute the Ricci form of the metric $g_\om$ defined by (\ref{kahlermetric}) in terms of the data on $M^4$. Using this we shall then characterise all $S^1$-invariant K\"ahler Einstein metrics on $6$-manifolds. 
\begin{Prop}\label{ricciformproposition}
	The Ricci form of $g_\om$, as defined by (\ref{kahlermetric}), is given in terms of the data on $M^4$ by 
	\begin{align*}
		-2 \Ric(g_\om) =\ &d_M d_M^c (\log (u^2 \tilde{\rho})) - u^{-4}\tilde{\rho}^{-1}\partial_H(u^2\tilde{\rho} )\cdot \partial_H(\tilde{\om}_1)  \\
		&+ d_M(u^{-4}\tilde{\rho}^{-1}\partial_H(u^2\tilde{\rho})) \w \Theta\\
		&+ u^2 dH \w  d_M^c(u^{-4}\tilde{\rho}^{-1}\partial_H(u^2\tilde{\rho}))\\
		&+ \partial_H(u^{-4}\tilde{\rho}^{-1}\partial_H(u^2\tilde{\rho}))\cdot dH \w \Theta
	\end{align*}
	where $\tilde{\rho}=\tilde{\rho}(H)$ denotes the exponential of the Ricci potential of $g_{\tilde{\om}_1}$ i.e. $$\Ric(g_{\tilde{\om}_1})=-\frac{1}{2}d_Md_M^c(\log \tilde{\rho}).$$ In particular, the scalar curvature of $g_{\om}$ is given by
	\[\partial_H(u^{-4}\tilde{\rho}^{-1}\partial_H(u^2\tilde{\rho}))-2\Big(d_M d_M^c (\log (u^2 \tilde{\rho})) - u^{-4}\tilde{\rho}^{-1}\partial_H(u^2\tilde{\rho} ) \partial_H(\tilde{\om}_1)\Big)_{\tilde{\om}_1},\]
	where the notation $(\al)_{\tilde{\om}_1}$ denotes the $\tilde{\om}_1$-component of an arbitrary $2$-form $\al$ on $M^4$ i.e. $(\al)_{\tilde{\om}_1}:=\frac{1}{2}g_{\tilde{\om}_1}(\alpha, \tilde{\om}_1)$. 
\end{Prop}
\begin{proof}
For a general K\"ahler $2n$-manifold the Ricci form is given by
\[\Ric(g_\om)=-\frac{1}{2} dd^c(\log \|\Upsilon\|^2_{g_{\om}}),\]
where $\Upsilon$ is a holomorphic $(n,0)$ form cf. \cite{KobayashiNomizu2}. 
On $P^6$, as given by Theorem \ref{kahlerreductiontheorem}, we can take
\[\Upsilon = (\Theta+iu^2dH)\w (dz_1 \w dz_2),\]
where $z_1,z_2$ denote local complex coordinates on $(M^4,J_1)$ so that 
$\tilde{\rho}:=\|dz_1 \w dz_2\|^2_{g_{\tilde{\om}_1}}$. Then we have that 
\[\Ric(g_\om)=-\frac{1}{2} dd^c(\log (u^2 \tilde{\rho}))\]
and the result follows by simplifying this expression. 
\end{proof}
We can now easily deduce the Einstein condition from Proposition \ref{ricciformproposition}:
\begin{Cor}\label{EinsteinConditionCorollary}
	The metric $g_\om$ as defined by (\ref{kahlermetric}) is Einstein i.e. $-2 \Ric(g_{\om})= \lambda \om$ if and only if
	\begin{gather}
		u^{-4}\tilde{\rho}^{-1}\partial_H(u^2 \tilde{\rho})= -C -\lambda H,\label{EinsteinCondition1}\\
		d_M d_M^c(\log (u^2 \tilde{\rho}))+(\lambda H +C) \partial_H(\tilde{\om}_1)=\lambda \tilde{\om}_1,\label{EinsteinCondition2}
	\end{gather}
	where $C$ is a constant.
\end{Cor}
We should emphasise that despite being able to describe all $S^1$-invariant Calabi-Yau metrics (at least locally), there are in general no known method for constructing the associated holomorphic volume form $\Om$ which is especially important in the study of special Lagrangians. This is likely the main reason why the construction of special Lagrangians has been focussed mainly to just $\C^n$ cf. \cite{Bryant2004SL, Haskins2004SL, Joyce2000SL}. Of particular interest to us are the examples arising from the so-called Calabi ansatz: 

\textit{Calabi's examples.}
The most well-known examples arising from Corollary \ref{EinsteinConditionCorollary} are the complete Calabi-Yau metrics on line bundles over Fano $4$-manifolds \cite{Calabi1957}: in this case $\lambda=0$, $\tilde{\om}_1=2H\om_1$ with $\om_1$ denoting the K\"ahler form on $M^4$ with $\Ric(\om_1)=6\!\ \om_1$ (here we have set $C=6$), $\tilde{\rho}=H^{-2}\rho_1$ with $\rho_1$ denoting the Ricci potential of $\om_1$ and 
\begin{equation}
	u=\frac{H}{\sqrt{2H^3-\check{C}}}
\end{equation} 
for some constant $\check{C}\geq 0$ (determining the size of the zero section $M^4\hookrightarrow P^6$). When $\check{C}=0$, $g_\om$ corresponds to a cone metric on a Sasaki-Einstein $5$-manifold. 
In this paper we are primarily interested in the case when $M^4$ is the homogeneous space $\C \mathbb{P}^2$ and $\C \mathbb{P}^1 \times \overline{\C\mathbb{P}^1}$ (and hence $P^6$ is a cohomogeneity one manifold). By means of a \textit{good} coordinate system we shall find below explicit expressions for the holomorphic volume form in these cases. 
\subsection{Constant solutions}\label{constantsolutionsection}
We now consider the simplest examples arising from Theorem \ref{kahlerreductiontheorem} when $u$ is solely a function of $H$ i.e. $d_Mu=0.$ Solving (\ref{kahlerPDE}) we get 
\begin{equation}
	\tilde{\om}_1=H\sigma_1 +\sigma_0,\label{generalomtildeconstantsolution}
\end{equation}
for some closed $(1,1)$ forms $\sigma_0$ and $\sigma_1$ on $M^4$, independent of $H$. 
We shall also assume that $(M^4,J_1)$ admits a standard K\"ahler form $\om_1$ and hence we can define a non-degenerate symmetric bilinear form $B$ on $2$-forms by wedge product:
\[\al \w \beta = B(\al,\beta)\!\ \om_1 \w \om_1 .\]
Restricting $B$ to the $2$-plane spanned by $\langle\sigma_0,\sigma_1 \rangle$ we may diagonalise so that the general solution to (\ref{generalomtildeconstantsolution}) can be expressed as
\begin{equation}
	\tilde{\om}_1= (aH+b)\om_1+(pH+q)\om_0,\label{constanttildeom}
\end{equation}
where $\om_0$ is a closed anti-self-dual form on $(M^4,J_1,\om_1,g_{\om_1})$ 
such that $\om_0 \w \om_0=-\om_1 \w \om_1$,
and $a,b,p,q$ are constants cf. \cite[Sect. 3]{Apostolov2003}. 
To ensure that $g_{\tilde{\om}_1}$ is indeed \textit{positive} definite we need to impose that
\[aH+b > |pH+q|,\]
otherwise we would be led to metrics with mixed signature i.e. the holonomy group will be contained in $U(1,2)$ rather than $U(3)$. We now have that
\[d\Theta = -a \om_1 -p \om_0  \in H^2(M^4,\mathbb{Z})\]
and (\ref{EinsteinCondition2}) becomes
\begin{equation}
	d_Md_M^c(\log(\tilde{\rho}))=(\lambda b -aC)\om_1+(\lambda q -p C)\om_0.\label{simplifyEinsteinCondition2}
\end{equation}
This shows that the Ricci form of $(M^4,g_{\tilde{\om}_1},J_1,\tilde{\om}_1)$ lies in the linear span of $\om_{0}$ and $\om_{1}$.
In particular, we see that if $p=q=0$ then $(M^4,J_1,\om_1,g_{\om_1})$ is in fact a K\"ahler Einstein manifold i.e.
\begin{equation}
	\Ric(\om_1)=-\frac{1}{2}(\lambda b -aC)\om_1\label{einsteinM4}
\end{equation}
with $\tilde{\rho}=(aH+b)^{-2}\rho_1$; $\rho_1$ being the exponential of the K\"ahler potential of $\om_1$. Hence solving (\ref{EinsteinCondition1}) we find that
\begin{equation}
	u=\frac{6(aH+b)}{\sqrt{9a^2\lambda H^4+12a(2b\lambda+aC)H^3+18b(b\lambda+2aC)H^2-36b^2 C H-\check{C}}},\label{solutionforuEinstein}
\end{equation}
where $\check{C}$ is a constant.
\begin{Rem}
	Observe that if $a=p=0$ so that $d\Theta=0$, then we get a product Einstein metric on $P^6=M^4 \times \Sigma^2$, where the universal cover of $\Sigma^2$ is isometric to either the round $2$-sphere, Euclidean $\R^2$ or the hyperbolic plane depending on the Ricci curvature of $(M^4,g_{\om_1})$. 
\end{Rem}
The family of, a priori only local, K\"ahler Einstein metrics determined by (\ref{einsteinM4}) and (\ref{solutionforuEinstein}) include several well-known complete 
Einstein metrics: 
\begin{enumerate}
	\item When $a=2,b=0,C=6, \check{C}>0, \lambda=0$ we recover the aforementioned complete Calabi-Yau metrics found by Calabi \cite{Calabi1957}: here $M^4$ is a Fano manifold.\label{calabiansatzgenral}
	\item When $a=1,b=0,C=12, \check{C}=0, \lambda=-16$ we obtain a positive K\"ahler Einstein metric on the $1$-point compactification of the hyperplane bundle of $M^4$: here $M^4$ is again a Fano manifold. e.g. if $M^4=\C\mathbb{P}^2$ then $P^6=\C\mathbb{P}^3=\overline{\mathcal{O}_{\C\mathbb{P}^2}(-1)}$ (with scalar curvature $48$). If $M^4$ is not $\C\mathbb{P}^2$ then we get an isolated conical singularity at infinity; in the $\C\mathbb{P}^2$ case there is no singularity since we have an $S^5$ collapsing smoothly to the point at infinity.\label{CP3asS1invariant}
	\item When $a=1, b=0,C=-12, \check{C}=0, \lambda=16$ we obtain complete negative K\"ahler Einstein \label{MoreEinstein}
	metrics which are again defined on complex line bundles since we need $H\geq 1$. Note that in this case $M^4$ is also a negative K\"ahler Einstein manifold. These metrics are the non-compact duals of case (\ref{CP3asS1invariant}).
	\item When $a>0, b=0,C=0, \check{C}=0, \lambda>0$ we again get complete negative K\"ahler Einstein metrics cf. \cite[Theorem 1.1 (1.5)]{UdhavFowdar4}, but now these are defined for all $H\in \R$. Note that in this case $M^4$ is a hyperK\"ahler manifold. \label{EinsteinonHK}
\end{enumerate}
A detailed analysis of when one gets complete non-compact K\"ahler Einstein metrics on $P^6$ can be found in \cite{HwangSinger1998}: the ODE analysis therein includes the case when $p,q$ are non-zero as well but we should however point out the authors do not rely on the explicit expression (\ref{solutionforuEinstein}) we gave here and instead study the behaviour of $u$ implicitly by dynamical methods. Aside from the above complete examples, there are also several incomplete but nonetheless interesting examples, which we describe next.

\subsubsection{Examples on $M^4=\mathbb{T}^4$.}\label{pqnonzeroexamples} 
When the constants $p,q$ are non-zero it is harder to identify the K\"ahler metric $g_{\tilde{\om}_1}$ explicitly in general. However if $M^4=\mathbb{T}^4$ then, without loss of generality by choosing suitable coordinates $x_i$, we can always take
\begin{gather*}
	\om_1=dx_{12}+dx_{34},\\
	\om_0=dx_{12}-dx_{34},
\end{gather*}
and hence
\begin{gather}
	\tilde{\om}_1=((a+p)H+(b+q)) dx_{12}+((a-p)H+(b-q)) dx_{34},\\
	g_{\tilde{\om}_1}=((a+p)H+(b+q))(dx_1^2+dx_2^2)+((a-p)H+(b-q))(dx_3^2+dx_4^2).
\end{gather}
The K\"ahler potential is now given by 
$$\tilde{\rho}=(((a+p)H+(b+q))((a-p)H+(b-q)))^{-1}\rho_1,$$ 
where in this case $\rho_1$ is just a constant. The total space of the $S^1$ bundle $Q^5$ is now a nilmanifold determined by
\[d\Theta = -(a+p) dx_{12} -(a-p) dx_{34}  \in H^2(\mathbb{T}^4,\mathbb{Z}),\]
which correspond, up to diffeomorphism, to either of the nilpotent Lie algebras
\[(0,0,0,0,12) \text{\ \ \  or \ \ \ } (0,0,0,0,12+34),\]
where we are using Salamon's notation to denote the Lie algebras cf. \cite{Salamoncomplexstructures}. The fact that these nilpotent Lie algebras can be used to construct Calabi-Yau metrics was also found by Conti-Salamon in \cite{ContiSalamon2007} by different methods.

From (\ref{simplifyEinsteinCondition2}) it now follows that we need
\[\lambda b -aC=\lambda q -p C =0\]
in order to get an Einstein metric.
If $C \neq 0$ then we can set $a=\lambda b/C$ and $p=\lambda q /C$. By rescaling the coordinates $x_1,x_2$ and $x_3,x_4$ it is not hard to see that we are led back to the case when $p=q=0$.

If instead $C=0$ then either $p=q=0$ and we are back to the above setting, or $\lambda=0$. In the latter case solving (\ref{EinsteinCondition2}) we get
\begin{equation}
	u=\sqrt{((a+p)H+(b+q))((a-p)H+(b-q))}.\label{uhalfcomplete}
\end{equation}
In this case, the Calabi-Yau metrics on $P^6$ are half-complete i.e. they are complete when $H\to +\infty$ but are incomplete at the other end of the interval on which $H$ is defined, see \cite[Sect. 6]{ContiSalamon2007} for similar nilmanifold examples but which are not bundles on $\mathbb{T}^4$. 

Similar half-complete metrics on $4$-manifolds were recently used in \cite{HeinSunViaclovskyZhang2022} to construct collapsing limits of $K3$ surfaces endowed with their \textit{complete} Calabi-Yau metrics. In the latter case however there is only one choice of a nilpotent Lie algebra, namely
\[(0,0,12)\]
and the K\"ahler quotient is of course just an elliptic curve $\mathbb{T}^2$. These half-complete metrics were shown to provide a good approximation for the asymptotic behaviour of the complete Tian-Yau metrics cf. \cite{TianYau}, so the metrics determined by (\ref{uhalfcomplete}) might find similar applications in future higher dimensional problems.

While this paper is mainly concerned with the case when $d_M u=0$, we shall now make a slight digression and discuss the case $d_M u \neq 0$ since to the best our knowledge this does not appear to have been studied in the literature.
\subsection{Non-constant solutions}\label{nonconstantsolutions}
In this section we want to show the simplest instance how one can construct deformations of some of the constant solutions Calabi-Yau metrics by allowing $u$ to vary on $M^4$. We follow the approach described in  \cite[Sect. 9]{UdhavFowdar2}.

We consider the case when $(M^4,g_M,\om_1,\om_2+i\om_3:=\sigma_2+i\sigma_3)$ is a hyperK\"ahler manifold together with a, possibly zero, closed anti-self-dual $\om_0$ as above. 
We look for solutions to (\ref{kahlerPDE}) and (\ref{uandomtilde}) of the form
\begin{equation}
	\tilde{\om}_1 = (aH+b)\om_1+(pH+q)\om_0 + d_M d_M^c G,
\end{equation}
where $G:\R_H \times M^4 \to \R$. When $G=0$ (so that $d_Mu=0$), from (\ref{uandomtilde}) we find that $u$ is given by (\ref{uhalfcomplete}) and we get again the aforementioned half-complete Calabi-Yau metrics. For a general function $G$ it is clear that $\tilde{\om}_1$ is still a K\"ahler form and that  the cohomology class $[\tilde{\om}_1]$ remains unchanged. 

Equation (\ref{kahlerPDE}) now becomes
\[d_Md_M^c(u^2-G'')=0,\]
where $'$ denotes partial derivative with respect to $H$, so without loss of generality we may take $u^2=G''$. Together with equation (\ref{uandomtilde}) we now have
\begin{align}
	G'' \om_1^2 = &((aH+b)^2-(pH+q)^2+(aH+b)\Delta_MG) \om_1^2\label{equationforG}\\
	&+2(pH+q)(d_Md_M^cG)\w \om_0 + (d_Md_M^cG)^2,\nonumber
\end{align}
where $\Delta_M$ denotes the Hodge Laplacian on $(M^4,g_M,\om_1,J_1)$.
The last term can more explicitly be written as
\[(d_Md_M^cG)^2=(\frac{1}{4}(\Delta_MG)^2-\frac{1}{2}\|(d_Md_M^cG)_0\|^2_{g_{\om_1}}) \om_{1}^2, \]
where $(d_Md_M^cG)_0$ denotes the traceless component of $d_Md_M^cG$ or equivalently its projection in $\Lm^2_-(M^4)$. Thus, this reduces the problem of finding a Calabi-Yau metric to solving a single second order PDE for the function $G$ only. 

Since $M^4$ is a hyperK\"ahler manifold we know that it admits a real analytic structure cf. \cite{DeTurckKazdan1981} and hence given real analytic initial data to (\ref{equationforG}) we  can appeal to the Cauchy-Kovalevskaya theorem to get a general existence result:
\begin{Prop}
	Given real analytic functions $G_0$ and $G_1$ on the hyperK\"ahler manifold $M^4$ such that $\tilde{\om}_1 = (aH+b)\om_1+(pH+q)\om_0 + d_M d_M^c G_0$ defines a positive definite metric. Then there exists a unique solution  $G(H)$ to (\ref{equationforG}) with $G(0)=G_0$ and $G'(0)=G_1$ defined for $H\in(0,T)$. Theorem \ref{kahlerreductiontheorem} then asserts that this defines (locally) a Calabi-Yau metric on $(0,T) \times S^1 \times M^4$.
\end{Prop}
\begin{Rem}
	In the language of exterior differential systems, the above proposition is saying that the $S^1$-invariant Calabi-Yau metrics are determined by $2$ functions of $4$ variables namely $G_0$ and $G_1$. Owing to the fact that we could reduce the problem of constructing Calabi-Yau metric to the single PDE (\ref{equationforG}) we were able to appeal to the Cauchy-Kovalevskaya theorem instead of the more general Cartan-K\"ahler theory; we refer the reader to \cite{BryantEmbedding} for a comparison with the Calabi-Yau $2$-fold case.
\end{Rem}
Since (\ref{equationforG}) is quite hard to solve in general, we shall make some further simplifying assumptions so that we can find certain explicit solutions.

From \cite[Theorem 2.4, 3.2]{Bedford77} we know that a smooth real function $F$ on $M^4$ satisfies 
\[(d_Md^c_MF)^2=0 \] 
if and only if $M^4$ admits a foliation by complex submanifolds, with the leaves corresponding to the integral complex curves of the ideal generated by $d_Md^c_MF$. In this case we may assume there exists locally a fibration $\pi:M^4 \to \Sigma^2 $, where $\Sigma^2$ is a complex curve and that $F$ descends to a function on $\Sigma^2$. Under this hypothesis on $G$ we can eliminate the quadratic term in (\ref{equationforG}) .

\textit{Examples on $\mathbb{T}^4$ again.} With $M^4=\mathbb{T}^4$ 
as above, we take $(a,b,p,q)=(1,0,0,0)$ and consider $G$ of the form
\[G(H,x_1,x_2)=v(H) F(x_1,x_2)+\frac{1}{12}H^4.\]
Here $\Sigma^2$ is the elliptic curve $\mathbb{T}^2$ with coordinates $(x_1,x_2)$. Equation (\ref{equationforG}) now reduces to the pair
\begin{gather}
	\Delta_M F = \mu F,\\
	v''=\mu H v,	
\end{gather}
where $\mu$ is a constant. For instance a simple solution is given by $\mu=1$, $F=\sin(x_1)$ and $v=\mathrm{Ai}(H)$ (the Airy function). Another solution is given by $\mu=-1$, $F=\exp(x_1)$ and $v=\mathrm{Ai}(-H)$. One can verify directly by computing the rank of the curvature operator, say using \textsc{Maple}, that indeed the resulting metrics have holonomy \textit{equal} to $SU(3)$. Similar deformations of $G_2$ and $Spin(7)$ holonomy metrics were found in \cite{Apostolov2003, UdhavFowdar2}.
\begin{Rem}
	Constructing Einstein deformations for the general case when $d_Mu=0$ is much harder since instead of (\ref{uandomtilde}) we are then required to solve (\ref{EinsteinCondition1}) and (\ref{EinsteinCondition2}) which also involves $\tilde{\rho}$ (itself depending on $\tilde{\om}_1$).
\end{Rem}
\section{$S^1$-invariant Hermitian Yang-Mills connections}\label{kahlerreductionofHYMsection}
Having fully encoded the K\"ahler Einstein condition on $P^6$ in terms of data on $M^4$ (by means of Theorem \ref{kahlerreductiontheorem} and Corollary \ref{EinsteinConditionCorollary}), we next proceed to do the same for the (deformed) Hermitian Yang-Mills conditions on $P^6$.
\subsection{K\"ahler reduction of the (deformed) Hermitian Yang-Mills condition}
Suppose now that 
\begin{equation}
	\underline{A}:=\psi \Theta+A\label{s1invtA}
\end{equation} 
is an $S^1$-invariant connection on $P^6$, where $A=A(H)$ denotes a family of connections on $E\to M^4$ and $\psi=\psi(H)\in \Gamma(\mathrm{End}(E))$ denotes a family of Higgs fields;
the $S^1$-invariant hypothesis means that $\mathcal{L}_X\underline{A}=0$ i.e.
\begin{equation}
	\mathcal{L}_XA =0\text{\ \ \ and\ \ \ } \mathcal{L}_X\psi=0.
\end{equation}
Note that in (\ref{s1invtA}) we assumed that $\underline{A}$ has no $dH$-term since we can always set it to zero by means of a suitable $S^1$-invariant gauge transformation, so indeed any $S^1$-invariant connection $\underline{A}$ can be expressed in the form (\ref{s1invtA}). Next we compute the curvature of $\underline{A}$.
\begin{Prop}
The curvature form $F_{\underline{A}}$ of (\ref{s1invtA}) is given by
\begin{equation} 
	F_{\underline{A}}=(d^A_M\psi+\partial_H \psi dH)\w \Theta +(\psi d_M^c(u^2)-\partial_H(A))\w dH+F_A-\psi \partial_H(\tilde{\om}_1),\label{s1invariantcurvatureform}
\end{equation} 
where $$d^A_M\psi:=d_M\psi + A \psi -\psi A$$ is the exterior covariant derivative of the Higgs field $\psi$ and $F_A:=d_MA +A\w A$ is the curvature of $A$ on $M^4$, depending on $H$. 
\end{Prop}
\begin{proof}
A simple computation shows that
\begin{align*}
	F_{\underline{A}} =\ &\partial_H\psi dH \w \Theta + d_M\psi \w \Theta+\psi d\Theta - \partial_H A \w dH +d_M A\\ &+ \psi \Theta \w A+\psi A \w \Theta + A \w A.
\end{align*}
It now suffices to use (\ref{curvaturekahlerreduction}) and this completes the proof.
\end{proof}
We are now in position to characterise the Hermitian Yang-Mills condition on $\underline{A}$ in terms of the data on $M^4$.
\begin{Th}\label{instantonkahlerreductiontheorem}
	The $S^1$-invariant connection form $\underline{A}:=\psi \Theta+A$ on $P^6$, as determined by Theorem \ref{kahlerreductiontheorem}, is a Hermitian Yang-Mills connection, say with gauge group $G$, if and only if on $M^4$ we have that $F_A \in \Lm^{1,1}(M^4)$ i.e. $A$ is Hermitian connection for each $H$ and the following holds:
	\begin{gather}
		u^2J_1(d^A_M\psi)=\partial_H(A)-\psi d^c_M(u^2),\label{instantonequ11}\\
		\partial_H(\psi)-2(F_A - \psi \partial_H(\tilde{\om}_1))_{\tilde{\om}_1} = C_0,\label{instantonequ12}
	\end{gather}
	where $C_0$ is a constant and $(F_A - \psi \partial_H(\tilde{\om}_1))_{\tilde{\om}_1}$ denotes the $\tilde{\om}_1$-component. 
\end{Th}
\begin{proof}
	First we have that (\ref{hym1}) becomes
	\begin{align*}
		0= 	F_{\underline{A}}\w \Om^+ = &-ud_M^A\psi\w \Theta \w dH \w \tilde{\om}_3+u^{-1}(\partial_H A-\psi d_M^c(u^2))\w \Theta \w dH \w \tilde{\om}_2\\
		&+(F_A-\psi \partial_H(\tilde{\om}_1)) \w (u^{-1}\Theta \w \tilde{\om}_2-udH \w \tilde{\om}_3)
	\end{align*}
	Equation (\ref{instantonequ11}) now follows from the fact that $\beta \w \tilde{\om}_3=J_1(\beta)\w \tilde{\om}_2$ for a general $1$-form $\beta$ on $M^4$. From Theorem \ref{kahlerreductiontheorem} we know that $\partial_H(\tilde{\om}_1) \in \Lm^{1,1}(M^4)$ and hence the same applies to $F_A$.
	A simple calculation now shows that (\ref{hym2}) holds if and only if
	\begin{equation}
		\partial_H(\psi)-2(F_A - \psi \partial_H(\tilde{\om}_1))_{\tilde{\om}_1} = C_0,
	\end{equation}
	and this concludes the proof.
\end{proof}
The Hermitian Yang-Mills condition on $\underline{A}$ on $P^6$ consists of $7 \dim(\mathfrak{g})$ equations. The constraint that $F_A \in \Lm^{1,1}(M^4)$ consists of $2 \dim(\mathfrak{g})$ equations, (\ref{instantonequ11}) of $4 \dim(\mathfrak{g})$ and (\ref{instantonequ12}) of $\dim(\mathfrak{g})$, so the dimensional count $7=2+4+1$ confirms that the above system is consistent. Observe also that if we set $\psi=0$ then the above system becomes equivalent to asking that $\underline{A}$ is (the pullback of) a Hermitian Yang-Mills connection on $M^4$. This is essentially due to the inclusion $F_{\underline{A}}=F_A \in \Lm^{1,1}(M^4)\cong \mathfrak{u}(2)\subset \mathfrak{u}(3)\cong \Lm^{1,1}(P^6)$.
If $d\Theta=0$ so that $P^6=\Sigma^2 \times M^4$ is a Riemannian product then the pair (\ref{instantonequ11}) and (\ref{instantonequ12}) reduce to a Bogomolny type equation. So Theorem \ref{instantonkahlerreductiontheorem} can be interpreted as a Bogomolny type equation twisted with curvature form $d\Theta$.
In the abelian case i.e. when $G=U(1)$, we always have that:
\begin{Cor}\label{killingvectorfieldcorollary}
	The connection form $\underline{A}:=X^{\flat}=u^{-2}\Theta$ on $P^6$ is a Hermitian connection i.e. $dX^\flat \in \mathfrak{u}(3) \cong \Lm^{1,1}=\langle \om \rangle \oplus \Lm^{1,1}_0\cong \mathfrak{u}(1)\oplus\mathfrak{su}(3)$. Moreover if $P^6$ is a Calabi-Yau $3$-fold, then $X^{\flat}$ is Hermitian Yang-Mills and the $\om$-component of $F_{\underline{A}}$ vanishes if and only if $k=0$ i.e. $\mathcal{L}_X\Om=0$.
\end{Cor}
\begin{proof}
	The first part is immediate from Theorem \ref{instantonkahlerreductiontheorem} by setting  $\psi=u^{-2}$ and $A=0$. The second part can be deduced from
	\[dX^\flat \w \om \w \om = -\frac{1}{2} d(JX \ip (\Om^+\w \Om^-))=-k (\Om^+ \w \Om^-)=-4k\vol_\om,\]
	where we used (\ref{normalisationsu3}) and Lemma \ref{permutingompluslemma}.
\end{proof}
\begin{Rem}
	We should emphasise that the second part of the corollary crucially relies on the fact that $P^6$ is a Calabi-Yau $3$-fold (even Einstein is not sufficient). For instance there exist examples of holomorphic Killing vector fields on $\C \mathbb{P}^3$ for which the $\om_{\C\mathbb{P}^3}$-component of $F_A$ is not constant and hence not Hermitian Yang-Mills.
\end{Rem}
Let us now consider the constant solution case i.e. when $d_Mu=0$ and $P^6$ is K\"ahler Einstein. The $S^1$-invariant Hermitian Yang-Mills condition of Theorem \ref{instantonkahlerreductiontheorem} then becomes:
\begin{Cor}\label{S1invariantInstantonconstantU}
	If $P^6$ is a K\"ahler Einstein manifold, as determined by Theorem \ref{kahlerreductiontheorem} with $d_Mu=0$, then $\underline{A}:=\psi \Theta+A$ defines an $S^1$-invariant Hermitian Yang-Mills connection on $P^6$ if and only if  $A$ defines a $1$-parameter family of Hermitian connections on $M^4$ satisfying
	\begin{align}
		  u^2J_1(d^A_M\psi) &= \partial_H(A),\label{eve1}\\
		(\partial_H(\psi)-C_0) ((aH+b)^2-(pH+q)^2)\!\ \om_1^2 &= +2(aH+b) F_A \w \om_1\nonumber\\ &\ \ \ +2(pH+q) F_A \w\om_0\label{eve2}\\
		&\ \ \  - 2\psi ((a^2-p^2)H+(ab-pq))\!\ \om_1^2.\nonumber
	\end{align}
	In particular, the curvature of $A$ on $M^4$ `evolves' by
	\begin{equation}
		\partial_H(F_A) = u^2 (d^A_M \circ J_1 \circ d^A_M) \psi.\label{evecurvature}
	\end{equation}
	Moreover in the case when $p=q=0$ we get
	\begin{equation}
		(aH+b) \partial^2_{H,H}(\psi)+3a  \partial_H(\psi)-aC_0= u^2 \Delta_A(\psi),\label{parabolicequforpsi}
	\end{equation}
	where 
	the connection Laplacian $\Delta_A$ is defined by $ \Delta_A(\psi) \!\ \om_{1}^2:=2( d^A_M \circ J_1 \circ d^A_M  \psi) \w \om_1$.
\end{Cor}
\begin{proof}
	Substituting $d_Mu=0$ and (\ref{constanttildeom}) in Theorem \ref{instantonkahlerreductiontheorem} gives (\ref{eve1}) and (\ref{eve2}). (\ref{evecurvature}) follows by differentiating $F_A$ with respect to $H$ and using (\ref{eve1}). (\ref{parabolicequforpsi}) follows by differentiating (\ref{eve2}) with respect to $H$ and using (\ref{evecurvature}).
\end{proof}
The above system consists of a coupled non-linear set of equations involving $A$ and $\psi$. 
In the simplest situation however when $G=U(1)$ we can decouple the equations and thus, one can hope to find explicit solutions which is what we investigate in the next section. 

We conclude this section with an analogous result for $S^1$-invariant deformed Hermitian Yang-Mills connections:
\begin{Th}\label{S1invariantdhymtheorem}
	If $P^6$ is a K\"ahler Einstein manifold, as determined by Theorem \ref{kahlerreductiontheorem} with $d_Mu=0$ and $b=p=q=0$, then $\underline{A}:=\psi \Theta+A$ defines an $S^1$-invariant deformed Hermitian Yang-Mills connection on $P^6$ if and only if  $A$ defines a $1$-parameter family of Hermitian connections on $M^4$ satisfying (\ref{eve1}) and
	\begin{align}
		0=\ &\frac{1}{2}a^2(\partial_H\psi)(\psi^2-H^2)-a^2H\psi+\frac{1}{2}(\partial_H \psi)(F_A^2)_{\om_{1}^2}\label{dhymreduced}\\
		&-u^2 (d_M\psi \w d_M^c\psi\w (F_A-a\psi\om_1))_{\om_1^2}\nonumber\\
		&+{a(H-\psi \partial_H \psi) (F_A \w \om_1)_{\om_1^2}}\nonumber
	\end{align}
	where the subscript $\om^2_1$ denotes the $\om^2_1$-component.
\end{Th}
\begin{proof}
	First since the gauge group is $U(1)$, we have that $d_M^A\psi=d_M\psi$ and $F_A=d_MA$. From the proof of Theorem \ref{instantonkahlerreductiontheorem} we also know that the condition that $\underline{A}$ is a Hermitian connection is equivalent to (\ref{instantonequ11}) and $F_A \in \Lm^{1,1}(M^4)$. It now suffices to rewrite (\ref{dhym}) using (\ref{s1invariantcurvatureform}) and (\ref{constanttildeom}) together with hypothesis that $d_Mu=0$ and $b=p=q=0$, which leads to (\ref{dhymreduced}).
\end{proof}
We shall consider the simplest solutions arising from Theorem \ref{S1invariantdhymtheorem} in Section \ref{dhymsection}.
\subsection{The abelian case}
We specialise Corollary \ref{S1invariantInstantonconstantU} to the case when the gauge group $G$ is $U(1)$ i.e. we look for $S^1$ invariant abelian Hermitian Yang-Mills connection on $P^6$. As we are only aware of explicit complete solutions in the case when $p=q=0$ we shall restrict to this situation. In particular, this applies to the examples (\ref{calabiansatzgenral})-(\ref{EinsteinonHK}) above.

Since $G=U(1)$ we have $[\psi,A]=0$ and it follows that $d^A_M\psi=d_M\psi$ and $\Delta_A(\psi)$ becomes the usual Hodge Laplacian $d^*d(\psi)$ on $(M^4,g_{\om_1})$. We furthermore set $C_0=0$ and look for separable solution with the Higgs field of the form 
\begin{equation}
	\psi=\kappa F,\label{psiansatz}
\end{equation}
where $\kappa$ is only a function of $H$ and $F$ is only a function on $M^4$. 
Then from (\ref{parabolicequforpsi}) we get:
\begin{equation}
	\frac{(aH+b)\!\ \partial^2_{H,H}(\kappa)+3a \!\ \partial_H(\kappa)}{u^2 \!\ \kappa}=  \frac{d^*d(F)}{F}=\mu,\label{separableabeliancase}
\end{equation}
where $\mu$ is a constant and as before $u$ is given by (\ref{solutionforuEinstein}). 
Since $d^*d(F)=\mu \!\!\ F$ it follows that $\mu$ is an eigenvalue of the Laplacian on $(M^4,g_{\om_{1}})$ with $F$ as the corresponding eigenfunction. In particular, we know that if $M^4$ is compact then $\mu$ lies in a discrete set in $[0,+\infty)$, depending only on the isometric class of $(M^4,g_{\om_{1}})$. Thus, to summarise we have shown that:
\begin{Prop}
	When the Higgs field $\psi$ is of the form (\ref{psiansatz}), $S^1$-invariant abelian $SU(3)$-instantons are determined by the spectrum of $(M^4,g_{\om_1})$.
\end{Prop}
It is easy to see, after redefining the coordinate $H$, that (\ref{parabolicequforpsi}) is a linear elliptic operator for $\psi$ (when $G$ is abelian), so with suitable boundary conditions one might be able to argue that the separable solutions (\ref{psiansatz}) in fact generate a basis for all abelian $SU(3)$-instantons. Our goal in this paper however is simply to construct explicitly such solutions so we shall not address the problem if our solutions are exhaustive.
We should point out that one might not always be able to find solutions $\kappa$ for all the values of $\mu$ in the spectrum of $M^4$, so only a distinguished subset of the spectrum will parametrise such $SU(3)$-instantons (see Remark \ref{CP3remark} below). 
The natural question of interest becomes for which values of $\mu$ can one find such $SU(3)$ instantons and given such $\mu$, what is the moduli space of solutions? 
In view of Corollary (\ref{killingvectorfieldcorollary}) we already know that when $P^6$ is a Calabi-Yau manifold there is at least one Hermitian Yang-Mills connection, which corresponds to $\mu=0$ if $\mathcal{L}_X\Om=0$.

The rest of this paper is mainly dedicated to explicitly describing the aforementioned $SU(3)$ instantons in cases when $M^4$ is a homogeneous space. Sections \ref{CP2section} and \ref{S2S2instantonsection} are concerned with the situation when $M^4=\C \mathbb{P}^2$ and $\C \mathbb{P}^1 \times \overline{\C \mathbb{P}^1}$ with $P^6$ as the corresponding canonical bundle with the Calabi ansatz metric and in Section \ref{T4section} we consider the other K\"ahler Einstein metrics arising from (\ref{CP3asS1invariant}), (\ref{MoreEinstein}) and (\ref{EinsteinonHK}). In Section \ref{dhymsection} we construct the simplest solutions arising from Theorem \ref{dhymtheorem}.

\begin{Rem}\label{CP3remark}
When $P^6=\C\mathbb{P}^3$ (see (\ref{CP3asS1invariant}) above) we know that all holomorphic line bundles are given by $\mathcal{O}_{\C\mathbb{P}^3}(k)$ for $k\in\mathbb{Z}$: these are determined by integer multiples of the Fubini-Study form which themselves correspond to the curvature forms of the Hermitian Yang-Mills connections. The Donaldson-Uhlenbeck-Yau correspondence asserts that these are all the abelian Hermitian Yang-Mills connections \cite{Donaldson1987,UhlenbeckYau1986}. Since $C_0=0$ only for the trivial line bundle, we already know that $\psi=0$ is the only solution to (\ref{separableabeliancase}) which gives a globally well-defined $SU(3)$-instanton. This shows that although locally, by Picard-Lindel\"of theorem, we know that there exist solutions $\kappa$ to (\ref{separableabeliancase}) these will in general not give rise to globally well-defined instantons on $P^6$. In fact when $P^6=\C\mathbb{P}^3$ one can show that the general solution $\kappa$ is explicitly given in terms of certain hypergeometric functions which will always have singularities (see sub-sect. \ref{thegeneralmucase} and \ref{instantononCP3} below).
\end{Rem}
\section{Abelian instantons on the canonical bundle of $\C\mathbb{P}^2$}\label{CP2section}
As a way of making the construction described in the previous section concrete, we shall now specialise to the case when $P^6$ is the canonical bundle of $M^4=\C\mathbb{P}^2$ endowed with the Calabi ansatz structure. We begin by giving coordinate description of the latter which will come in handy later on.

\subsection{The Calabi-Yau metric on $\C^3$}\label{calabiyauonC3section}
The flat metric on $\R^6$ viewed as a conical metric can be written as
\[g_{\R^6}=dr^2 + r^2  g_{S^5},\]
where $g_{S^5}$ denotes the round metric on $S^5$.
In view of the Hopf fibration $S^1 \hookrightarrow S^5 \to \C \mathbb{P}^2$ we can define a connection $1$-form $\theta$ on $S^5$ such that $d\theta = 2\om_{\C\mathbb{P}^2}$ and then we have
\begin{gather}
	g_{\R^6}=dr^2 + r^2  (\theta^2 + g_{\C\mathbb{P}^2}),\\
	\om_{\R^6}= r dr \w \theta + r^2 \om_{\C\mathbb{P}^2}. 
\end{gather}
Here the subscript $\C\mathbb{P}^2$ refers to the standard Fubini-Study K\"ahler structure  (normalised so that $g_{\C\mathbb{P}^2}$ has scalar curvature $24$). Note that $\C\mathbb{P}^2$ is just the K\"ahler reduction of $\R^6\cong\C^3$ by the diagonal action of $U(1)\subset U(3)$. Viewing $\C \mathbb{P}^2$ as a cohomogeneity one manifold under the action of $SU(2)$ we can rewrite the above as
\begin{gather*}
	g_{\R^6}=dr^2 + r^2  (\theta^2 +\frac{1}{4t(1-t)}dt^2+ t(1-t)\sigma_1^2 + t(\sigma_2^2+\sigma_3^2)),\\
	\om_{\R^6}= r dr \w \theta + r^2  \big( \frac{1}{2} \sigma_1 \w dt + t \sigma_2 \w \sigma_3\big). 
\end{gather*}
where $\theta:=dy-t\sigma_1$, with $y$ denoting the local coordinate on the Hopf $S^1$ fibre, and $\sigma_i$ are the usual left invariant coframe on $S^3$ satisfying $d\sigma_1=-2 \sigma_{23}$ (+ cyclic permutation of the indices).
It will be useful for later (see sub sect. \ref{themu12case}) to also have the following more concrete description: using Euler angles $(x_1,x_2,x_3)$ on $S^3$ we can express the above left invariant $1$-forms as
\begin{gather*}
	\sigma_1 = \frac{1}{2} (\cos(x_1) dx_2+\sin(x_1)\sin(x_2) dx_3),\\
	\sigma_2 = -\frac{1}{2} (\sin(x_1) dx_2-\cos(x_1)\sin(x_2) dx_3), \\
	\sigma_3 = \frac{1}{2} ( dx_1+\cos(x_2) dx_3).
\end{gather*}
Observe that there are two singular orbits for this cohomogeneity one action: an $S^2$ (when $t=1$ corresponding to the collapsing of the $S^1$ generated by the vector field $\partial_{\sigma_3}$) and a point (when $t=0$ corresponding to the collapsing of the principal orbit $S^3$ itself). The latter is easily seen by setting $t=\cos^2(s)$ so that
\[ds= -\frac{dt}{2 t^{1/2}(1-t)^{1/2}}\]
and hence we have
\begin{align*}
	g_{\R^6} =&\ dr^2 + r^2\theta^2 + r^2 \big(ds^2+\sin^2(s)\cos^2(s) \sigma_1^2
	+\cos^2(s) (\sigma_2^2+\sigma_3 ^2).
\end{align*}
We denote by $\partial_{\sigma_i}$ the left-invariant vector field on $S^3$ dual to $\sigma_i$. 
{Here the local coordinate $y$ on the $S^1$ fibre is required to have period $2\pi$. 
If $y$ had period $2\pi/k$ for $k \in \mathbb{Z}$ then the asymptotic topology would be that of $\mathcal{O}_{\C\mathbb{P}^2}(-k)$ i.e. $\C^3/\mathbb{Z}_k$ cf. \cite{PagePope1985}.}
\subsection{The Calabi-Yau metric on $\mathcal{O}_{\mathbb{C}\mathbb{P}^2}(-3)$}
As seen in subsection \ref{constantsolutionsection} the Calabi ansatz more generally yields the asymptotically conical Calabi-Yau metric
\begin{equation}
	g_{K}=\Big(\frac{r^6}{r^6-C}\Big) dr^2 + \Big(\frac{ r^6 -C}{r^4}\Big) \theta^2 + r^2 g_{\C\mathbb{P}^2}.\label{calabiyaumetric}
\end{equation}
on $K=\mathcal{O}_{\mathbb{C}\mathbb{P}^2}(-3)$ and we already noted above that when $C=0$ we recover the flat metric. In contrast to the flat metric however $g_{K}$ has holonomy group \textit{equal} to $SU(3)$. 

Note that the symplectic form is the same irrespective of the choice of the constant $C$ i.e. the associated K\"ahler form is still given by
\begin{equation}
	\om_{K}= r dr \w \theta + r^2  \Big( \frac{1}{2} \sigma_1 \w dt + t \sigma_2 \w \sigma_3\Big). \label{omCY}
\end{equation}
We can now define a natural constant length $(3,0)$-form by
\[\psi^++i\psi^-= \big(\frac{r^5 dr}{(r^6-C)^{1/2}}+i (r^6-C)^{1/2}\theta \big)\w ( \frac{dt}{2(1-t)^{1/2}}-i t (1-t)^{1/2} \sigma_1 )\w (\sigma_2 +i \sigma_3 ).\]
The latter is however not holomorphic, but since $g_{K}$ is a Calabi-Yau metric we know that there exists an orthonormal rotation of $\psi^\pm$ which defines a pair of parallel $3$-forms i.e. there is a real function $f$ such that $e^{if}(\psi^++i\psi^-)$ is parallel, or equivalently closed cf. Theorem \ref{su3torsionthm}. 
Solving the latter first order PDE we find: 
\begin{Prop}\label{holomorphicvolumeformCP2}
	The pair
	\begin{gather}
		\Om^+=\cos(3y)\psi^+ - \sin(3y) \psi^-,\\
		\Om^-=\sin(3y)\psi^+ + \cos(3y) \psi^-,
	\end{gather}
	defines a holomorphic volume form i.e. $d\Om^\pm=0$ and satisfies conditions (\ref{compatibilitysu3}) and (\ref{normalisationsu3}).
\end{Prop}
\begin{proof}
	It suffices to verify that $\Om^\pm$ as defined in the proposition are indeed closed. 
	Computing we find that
	\begin{align*}
		d\psi^{\pm} &= \pm3 (\theta+t\sigma_1) \w \psi^{\mp},\\
		&= \pm 3\ \! dy \w \psi^{\mp},
	\end{align*}
	where we used the definition of $\theta$ for the last equality. The result now follows immediately from this.
\end{proof}
To the best of our knowledge the above explicit expression for the holomorphic volume form was previously unknown. 
{Observe that in contrast to case when $P^6=\C^3$, we now require $y$ to have period $2\pi/3$ and hence the Calabi-Yau metric will be globally well-defined on $\mathcal{O}_{\C \mathbb{P}^2}(-3)$ i.e. the canonical bundle of $\C\mathbb{P}^2$.}
Thus, we have given an explicit description of the $SU(3)$-structure on $\mathcal{O}_{\C \mathbb{P}^2}(-3)$ arising from the Calabi ansatz. This will allow us later on to construct special Lagrangians on $\mathcal{O}_{\C\mathbb{P}^2}(-3)$. We should emphasise that the rather simple nature of the calculation in the above proof is due to the fact that we chose a good coordinate system on $\C\mathbb{P}^2$ to express the K\"ahler structure (which was in fact done in hindsight to simplify the proof). In general solving the first order PDE for $f$ explicitly can be rather involved.
\begin{Rem}
Note that although $\partial_y$ is a holomorphic Killing vector field  
it does not preserve the holomorphic volume form $\Om$: $\mathcal{L}_{\partial_y}\Om^\pm=\mp3\Om^\mp$. This illustrates Lemma \ref{permutingompluslemma}.
\end{Rem}
\subsection{Instantons: Abelian examples}
We now want to look for abelian instantons which are invariant under the $S^1$ action generated by $X=\partial_y$, but before doing so we first describe a few simple motivating examples.
\subsubsection{Elementary examples}
It is easy to verify that the following are all abelian $SU(3)$-instantons on $\mathcal{O}_{\C\mathbb{P}^2}(-3)$: 
\begin{align}
	\underline{A}= r^{-4}\theta\ &\Longrightarrow \ \|F_A\|^2_{g_{K}}=24 r^{-12},\label{firstone}\\
	\underline{A}= t^{-1}\sigma_1\ &\Longrightarrow \ \|F_A\|^2_{g_{K}}=8 t^{-4}r^{-4},\label{pullbackexample1}\\
	\underline{A}= (1-t^{-1}) \sigma_2\ &\Longrightarrow \ \|F_A\|^2_{g_{K}}=8 (1-t) t^{-4}r^{-4}.\label{pullbackexample2} 
\end{align}
Observe that examples (\ref{pullbackexample1}) and (\ref{pullbackexample2}) are both pullbacked from $\C \mathbb{P}^2$, i.e. $\psi=0$, and moreover blow-up when $t=0$ i.e. over the point singular orbit of $\C \mathbb{P}^2$ (viewed as a cohomogeneity $1$ manifold as described above). Put differently these correspond to meromorphic connection forms on $\C\mathbb{P}^2$ with isolated pole at $t=0$. Hence as instantons on $P^6$ they blow up along the $\C$ fibre over this point which is a holomorphic curve (i.e. calibrated by $\om_K$). This is consistent with the results in \cite{Tian2000} which assert that the bubbling of instantons occurs along calibrated submanifolds, although strictly speaking our instantons have infinite Yang-Mills energy and $P^6$ is non-compact. These examples could be considered as the possible limits.

By contrast consider now the example given by (\ref{firstone}) which is not pullbacked from $\C\mathbb{P}^2$ and is actually well-defined for $r \in [C^{1/6},\infty)$. 
{Note however that for the flat metric i.e when $C=0$ this instanton has a singularity at the origin.} Its curvature form is given by
\[F_{\underline{A}}=2r^{-6} (-2r dr \w \theta + r^2  \om_{\C\mathbb{P}^2})\]
and hence one finds that
\begin{align*}
	\mathrm{YM(\underline{A})} &= \int_{\mathcal{O}_{\C\mathbb{P}^2}(-3)} \|F_{\underline{A}}\|^2_{g_K} \vol_{K}\\
	&= 24 \!\ \frac{2\pi}{3} \!\ \vol(\C\mathbb{P}^2)  \!\ \int_{C^{1/6}}^{+\infty} r^{-7}dr \\
	&= \frac{8 \pi}{3 C}\!\ \vol(\C\mathbb{P}^2).
\end{align*}
The energy of the instanton concentrates/blows up at the origin in $\C^3$ i.e. the vertex of the cone while on the canonical bundle it is distributed on the zero section $\C \mathbb{P}^2$ (which is calibrated by $\om_K^2/2$). Thus, this instanton essentially distinguishes the asymptotically Euclidean metrics $g_{K}$ and $g_{\R^6}$. In fact this is the case for all such Calabi ansatz metrics i.e. the energy of the instanton (\ref{firstone}) corresponds to the size of the Fano manifold $M^4\hookrightarrow P^6$ viewed as the zero section. 

Aside from $SU(3)$ instantons, one can consider general Hermitian Yang-Mills connections.
Here are $2$ simple examples illustrating Corollary \ref{killingvectorfieldcorollary}.
The vector fields $\partial_y$ and $\partial_{\sigma_1}$ are both holomorphic Killing vector fields and neither preserve $\Om$. Hence we have:
\begin{gather*}
	d(\partial_y)^\flat \w \Om^+ =0,\\
	d(\partial_y)^\flat \w \om_{K} \w \om_{K} = +2 \!\ \om_{K} \w \om_{K} \w \om_{K},
\end{gather*}
and 
\begin{gather*}
	d(\partial_{\sigma_1})^\flat \w \Om^+ =0,\\
	d(\partial_{\sigma_1})^\flat \w \om_{K} \w \om_{K} = -\frac{4}{3}\!\ \om_{K} \w \om_{K} \w \om_{K}.
\end{gather*}
Next we consider general $SU(3)$-instantons.
\subsubsection{The $S^1$-invariant instanton condition}\label{S1invariantInstantonAntiCanonical}
We now specialise Corollary \ref{S1invariantInstantonconstantU} to our present context. We modify slightly the coordinates for convenience and consider $\underline{A}$ of the form
\[\underline{A}=f(r)\!\  \theta + A(r),\]
where in the notation introduced in the previous section $\Theta=-\theta$, $H=r^2/2$, $\tilde{\om}_1=r^2 \om_{\C \mathbb{P}^2}$ and $\psi=-f$; so $a=2$ and $b=p=q=0$. The $S^1$-invariant $SU(3)$-instanton condition (\ref{eve1}) and (\ref{eve2}) then becomes:
\begin{gather}
	\frac{\partial A}{\partial r}= - \frac{r^5}{r^6-C}  I (d^A_Mf),\label{instantonequ1}\\
	2 (F_A)_{\om_{\C\mathbb{P}^2}} + 4f + r \frac{\partial f}{\partial r} =0,\label{instantonequ2}
\end{gather}
where $I$ denotes the standard Fubini-Study complex structure and as before $A(r)$ is a $1$-parameter family of Hermitian connection on $\C \mathbb{P}^2$. In particular, we have: 
\begin{gather}
	\frac{\partial}{\partial r}(F_A) = -\frac{r^5}{r^6-C} (d^A_M \circ I \circ d^A_M f),\\
	\frac{r^5}{r^6-C} \Delta_A(f) = r \frac{\partial^2 f}{\partial r^2} + 5 \frac{\partial f}{\partial r} . \label{fundamentalpde}
\end{gather}

Note that equation (\ref{fundamentalpde}) is not in general linear in $f$ since the connection Laplacian $\Delta_A$ itself depends on $f$ via (\ref{instantonequ1}), except when $G$ is abelian.
So considering the case when $G=U(1)$ and looking for solution of the form 
\begin{equation}
	f=\kappa(r)\!\ F,\label{ansatz}
\end{equation} 
where $\kappa(r)$ is only a function of $r$ and $F$ only a function on $\C \mathbb{P}^2$, we find that  
(\ref{separableabeliancase}) becomes
\begin{equation}
	(r^6-C) \Big( \frac{5\kappa'+r \kappa''}{r^5  \kappa}\Big)= \frac{d^*dF}{F} =\mu.\label{fundamentalode}
\end{equation}
It is well-known that the eigenvalues of the Laplacian on $\C\mathbb{P}^2$ are given by 
$$\mu_k=4k(k+2)$$
for $k \in \mathbb{Z}^+_0$ and have multiplicities 
$$m_{k}= 4(k+1)^3$$
for $k\in \mathbb{Z}^+$ cf. \cite[Chap. 3 Proposition C.III.1]{Berger1971spectre}. So we only need to solve for $\kappa(r)$ when $\mu=4k(k+2)$. From (\ref{instantonequ1}) the connection $\underline{A}$ is then given by
\begin{equation}
	\underline{A}=\kappa  F  \theta- \Big(\int \frac{r^5 \kappa}{r^6-C}\ dr\Big) (d^cF)+A_0,\label{Abarsolution}
\end{equation}
where $A_0$ is a Hermitian connection on $\C \mathbb{P}^2$. From (\ref{instantonequ2}), and using the fact that $\kappa$ and $F$ solve (\ref{fundamentalode}), it follows that we need
\begin{equation}
	(d A_0)_{\om_{\C\mathbb{P}^2}}=c_0  F,\label{A0condition}
\end{equation}
where $c_0$ is a constant. If $\mu \neq 0$ then $c_0$ is determined by the choice of constant of integration in (\ref{Abarsolution}). So henceforth we shall assume that the constant on integration has been fixed so that $A_0=0$. 
The curvature form is then given by
\begin{align}
	F_{\underline{A}}=\ &\kappa' F dr \w \theta - \Big(\int \frac{r^5 \kappa}{r^6-C}\ dr\Big)dd^cF+\kappa dF \w \theta\label{curvatureabelian}\\
	& -\frac{r^5 \kappa}{r^6-C}\ dr \w d^cF+2\kappa F \om_{\C\mathbb{P}^2}\nonumber
\end{align}
and in view of Chern-Weil theory it defines a cohomology class in $H^2(P^6,\mathbb{Z})\cong H^2(\C\mathbb{P}^2,\mathbb{Z})\cong \mathbb{Z}$.
Note that for the $U(1)$ gauge group, two connections $\underline{A}_1$ and $\underline{A}_2$ are gauge equivalent if and only if $\underline{A}_2-\underline{A}_1$ is exact and hence $F_{\underline{A}_1}=F_{\underline{A}_2}$. An inspection of (\ref{curvatureabelian}) shows that this cannot happen if $F$ are distinct eigenfunctions, so indeed (\ref{Abarsolution}) determines distinct instantons.
\begin{Rem}
	Note that the above computations apply to any Calabi-Yau $3$-fold arising from the Calabi ansatz by replacing $\C \mathbb{P}^2$ with a suitable Fano $4$-manifold (see (\ref{calabiansatzgenral}) above) and hence changing $\mu$ accordingly. We consider explicitly the case when $M^4=S^2 \times S^2$ in the Section \ref{S2S2instantonsection} below.
\end{Rem}
\subsubsection{The $\mu=0$ case}\label{k=0case}
If $\mu=0$ then without loss of generality we can set $F=1$, and solving for $\kappa$ we find
\[\kappa = \frac{c_1}{r^4}+c_2.\]
Observe that in this case the solution is independent of $C$. The connection form $\underline{A}$ is then given by
\[\underline{A}= \Big(\frac{c_1}{r^4}+c_2 \Big) \theta + A_0,\]
where $A_0$ is a Hermitian connection on $\C \mathbb{P}^2$ independent of $r$. Substituting in (\ref{instantonequ2}) we find that we need
\[(F_{A_0})_{\om_{\C\mathbb{P}^2}}=(dA_0)_{\om_{\C\mathbb{P}^2}}=-2 c_2\]
i.e. $A_0$ is a Hermitian Yang-Mills instanton on $\C \mathbb{P}^2$. So if $c_1=0$ then $F_{\underline{A}}=(F_A)^{1,1}_0 \in \mathfrak{su}(2) \subset \mathfrak{su}(3)$ i.e. we are reduced to anti-self-dual instantons on $\C\mathbb{P}^2$. On the other hand if $c_2=0$ then we recover example (\ref{firstone}) above: moreover this is the only $S^1$ invariant abelian instanton with finite Yang-Mills energy as we shall see in sub-sect. \ref{thegeneralmucase} below.

\subsubsection{The $\mu=12$ case}\label{themu12case}
Next we consider the first non-trivial eigenvalue when $k=1$ so that $\mu=\mu_1=12$. In the above local coordinates description of $\C\mathbb{P}^2$ one can directly verify that the functions
\begin{gather*}
	F_1 = t  (\cos(x_3)\cos(x_1)-\sin(x_3)\sin(x_1)\cos(x_2)),\\
	F_2 = t  (\sin(x_3)\cos(x_1)+\cos(x_3)\sin(x_1)\cos(x_2)),\\
	F_3 = -t  \sin(x_1)\sin(x_2),
\end{gather*}
are all eigenfunctions on $\C\mathbb{P}^2$ with eigenvalue $12$ (there are of course more eigenfunctions but these were the only ones we were able to find explicitly in the above coordinates). Solving for $\kappa(r)$ we find the general solution:
\begin{align}
	\kappa(r) =\ c_1 &\frac{r^6-C}{r^4} + c_2 \frac{r^6-C}{r^4}  \big(2\sqrt{3} \arctan\Big(\frac{2 r^2}{\sqrt{3}\ C^{1/3}}+\frac{1}{\sqrt{3}}\Big)\label{k=1solution}\\ 
	&+ 3\log (r^2-C^{1/3})-\log(r^6-C)\big)+ 6 c_2 \!\ C^{1/3}.\nonumber
\end{align}
When $c_2=0$, the connection $\underline{A}$ determined by $\kappa(r)$ and $F_1, F_2, F_3$ by expression (\ref{Abarsolution}) (with $c_0$ chosen so that $A_0=0$) correspond to the dual $1$-forms associated to the Killing vector fields:
\begin{gather*}
	\frac{\sin(x_3)}{\sin(x_2)}\!\ \partial_{x_1}+\cos(x_3)\!\ \partial_{x_2}-\frac{\cos(x_2)\sin(x_3)}{\sin(x_2)}\!\ \partial_{x_3},\\
	\frac{\cos(x_3)}{\sin(x_2)}\!\ \partial_{x_1}-\sin(x_3)\!\ \partial_{x_2}-\frac{\cos(x_2)\cos(x_3)}{\sin(x_2)}\!\ \partial_{x_3},\\
	\partial_{x_3},
\end{gather*}
which generate \textit{right invariant} $S^1$ actions on the principal $S^3$ orbit on $\C\mathbb{P}^2.$ These Killing vector fields also preserve $\Om$ and hence $dX^\flat \in \mathfrak{su}(3)$, see Corollary \ref{killingvectorfieldcorollary}. In fact it is the knowledge of these Killing vector fields that led us to find explicitly the eigenfunctions $F_i$. Note that by contrast the \textit{left invariant} vector fields are not even Killing, with the exception of $\partial_{\sigma_1}$ which we already saw above.

\subsubsection{The general $\mu=\mu_k=4k(k+2)$ case.}\label{thegeneralmucase} 
We now want to find the general solution $\kappa$ for all the eigenvalues of $\C\mathbb{P}^2$.

Assuming $C\neq 0$ we define, with hindsight, a new variable $z=r^6/C$ and set $\kappa=C(1-z)f(z)$. Substituting in (\ref{fundamentalode}) we find that $f(z)$ satisfies:
\begin{equation}
	z(1-z)\frac{d^2f}{dz^2}+\Big(\frac{5}{3}-\Big(\frac{3-k}{3}+\frac{k+5}{3}+1\Big)z\Big)\frac{df}{dz}-\frac{(3-k)(k+5)}{3^2}f=0,\label{hypergeometricODE}
\end{equation}
which is the well-known hypergeometric differential equation. Observe that (\ref{hypergeometricODE}) has singular points at $z=0$ and $1$. However they are both regular singular points and hence one can always find local power series solutions in the neighbourhood of these points by the Frobenius method. Note that for general singular points of ODEs one cannot guarantee the existence of analytic solutions.
\begin{Rem}
In our case since we are considering a $U(1)$ gauge group (with an $S^1$ symmetry on $P^6$ and a suitable ansatz for $\psi$) we only have to deal with a single ODE, but for more general gauge groups (with additionally cohomogeneity symmetry on the base) one is usually led to a system of ODEs and then one has to appeal to Malgrange's method to get local existence of an analytic solution, see for instance \cite[Sect. 3.3, 4.3]{Lotay2016}. Given such a local solution one then need to show that it can be extended globally by means of ODE analysis. 
\end{Rem}
In our situation we are interested in solutions defined for $z\in [1,\infty)$ i.e. $r \in [C^{1/6},\infty)$. 
Consider instead the general solution to (\ref{hypergeometricODE}) in the neighbourhood of $r=0$ given by
\begin{align}
	\kappa(r) =\ &c_1  (r^6-C)\ _2F_1\Big(\frac{3-k}{3},\frac{k+5}{3};\frac{5}{3};\frac{r^6}{C}\Big)\ +\label{kappasolutionCP2}\\ 
	&c_2 \frac{ (r^6-C)}{r^4}\ _2F_1\Big(\frac{1-k}{3},\frac{k+3}{3};\frac{1}{3};\frac{r^6}{C}\Big),\nonumber
\end{align}
where $_2F_1$ denotes the hypergeometric function and $c_1,c_2$ are constants. 
We refer the reader to \cite[Chap. 15]{HandbookofMathematicalFunctions} for general properties of the hypergeometric function.
In general the hypergeometric function $_2F_1(a,b;c;z)$ has radius of convergence equal to $1$ unless either one of the first two arguments of $_2F_1$ is a non positive integer, say $a\in \mathbb{Z}^-$, in which case it reduces to a polynomial of degree equal to $-a$.
So if $k=3n$ or $3n+1$ then $\kappa(r)$ always has at least one solution defined for $r\in[C^{1/6},\infty)$; we already saw the $n=0$ case above (i.e. $k=0$ and $1$) and for instance if $n=1$ we have
\[\kappa(r)=r^6-C\]
and 
\[\kappa(r)=\frac{(r^6-C)(7r^6-C)}{r^4}\]
as explicit solutions corresponding to $k=3$ and $k=4$ respectively. Hence we obtain global existence of a solution when $k=3n$ and $3n+1$. 
The general solution to (\ref{hypergeometricODE}) in the neighbourhood of $r=C^{1/6}$ i.e. $z=1$ is generated by 
$$_2F_1(\frac{3-k}{3},\frac{k+5}{3};2;1-z)=z^{-2/3}\   _2F_1(\frac{1-k}{3},\frac{k+3}{3};2;1-z)$$ 
and the second fundamental solution has a more complicated power series expression, see \cite[15.5.19]{HandbookofMathematicalFunctions}, which generally converges for $|1-z|<1$ unless $k=3n$ or $3n+1$ in which case one gets a solution involving $\arctan$ and $\log$ as we saw already in (\ref{k=1solution}) [we were unable to find an explicit expression in the general case]. Thus, if $k=3n+2$ then we do not expect to find any complete solutions.

On the other hand if $C=0$ i.e. we are on $\C^3$ then solving (\ref{fundamentalode}) we instead get the general solution
\[\kappa(r)=c_1 r^{2k}+c_2 r^{-4-2k}.\]
Thus in this case there is always a solution for any $k\in \mathbb{Z}^+_0$. Of course, only the solution with $c_2=0$ will be well-defined on all of $\C^3$; when $c_2 \neq 0$ we get meromorphic connections with poles at the origin. Note that the case $\mu=c_2=0$ only gives a trivial solution as already seen above. So to summarise we have shown that:
\begin{Th}\label{maintheoremcanonicalbundle}
	For each positive integer $k$ there exist non-trivial $S^1$-invariant abelian $SU(3)$-instantons on $\C^3$ given by (\ref{Abarsolution}), where $F$ is an eigenfunction of $\C\mathbb{P}^2$ with eigenvalue $\mu_k=4k(k+2)$ and $\kappa=r^{2k}$.
	If $k$ is of the form $3n$ or $3n+1$ for $n\in \mathbb{Z}^+_0$, then there exist non-trivial  $S^1$-invariant  abelian $SU(3)$-instantons on $\mathcal{O}_{\C\mathbb{P}^2}(-3)$, where $\kappa$ is now given by (\ref{kappasolutionCP2}) for suitable constants $c_1,c_2$ (not both zero).  In particular, given a solution $\kappa$ for $\mu_k$, the moduli space of such instantons is a vector space of dimension equal to $m_{k}= 4(k+1)^3$.
\end{Th}
The result of this section more generally implies that the moduli space of abelian $S^1$ invariant $SU(3)$-instantons on the line bundle $P^6\to M^4$ endowed with the Calabi ansatz metric is modelled on the space of eigenfunctions of the Laplacian of $(M^4,g_M)$, provided there exists a solution $\kappa$ to (\ref{fundamentalode}) for $r\in[C^{1/6},\infty)$.
Note that the eigenvalue $\mu$ determines the solution $\kappa$ and hence the asymptotic growth rate of $F_{\underline{A}}$. From the discussion following (\ref{kappasolutionCP2}) we see that one can have polynomial growth rate $r^{l}$ for arbitrarily large integer $l$ and that (\ref{firstone}) is the only case when we have an instanton of finite Yang-Mills energy. 

We conclude this sub-section with the description of the Levi-Civita connection as an $S^1$ invariant \textit{non-abelian} $SU(3)$-instanton.
\subsubsection{The Levi-Civita connection as an $SU(3)$ instanton}
Finding non abelian instantons is a rather challenging problem and unfortunately we have not been able to find any new examples.
There is always however one canonical $SU(3)$-instanton with gauge group $SU(3)$ itself, namely the Levi-Civita connection.
With respect to a suitable orthonormal framing, one can express the Levi-Civita connection of $g_K$ locally as $\underline{A}=\psi\theta+A$ where
\[\psi=-\frac{iC}{r^6}\begin{pmatrix}
	{2} & 0 & 0 \\
	\cdot  & -1 & 0 \\
	\cdot  & \cdot & -1 
\end{pmatrix}\!\]
and 
\[A=\begin{pmatrix}
	-i t\sigma_1 & -\frac{\sqrt{t(r^6-C)}}{r^3}\sigma_2+i\frac{\sqrt{t(r^6-C)}}{r^3}\sigma_3 & -\frac{\sqrt{t(1-t)(r^6-C)}}{r^3}\sigma_1+i\frac{\sqrt{(r^6-C)}}{2r^3\sqrt{t(1-t)}}dt \\
	\cdot  & i\sigma_1 & \sqrt{1-t}\sigma_3-i\sqrt{1-t}\sigma_2 \\
	\cdot  & \cdot & i(t-1)\sigma_1 
\end{pmatrix}\!.\]
We see from this that $\psi$ is of a rather simple form and is only a function of $r$ i.e. $d_M \psi=0$; this might provide a useful ansatz to search for other examples in future work. Note that when $C=0$, $\underline{A}$ reduces to the flat connection. Note also that the Levi-Civita connection of $\C \mathbb{P}^2$ is itself a $U(2)\cong U(1){SU}(2)$ Hermitian Yang-Mills connection on $\C \mathbb{P}^2$ given locally by 
\[\begin{pmatrix}
	i(2t-1)\sigma_1 & -\sqrt{1-t}\sigma_3-i\sqrt{1-t}\sigma_2\\
	\cdot  & i(t+1)\sigma_1
\end{pmatrix}\!.\]
The $\mathfrak{su}(2)$-component naturally pulls back to give an instanton on $\mathcal{O}_{\C\mathbb{P}^2}(-3)$ under the inclusion $\Lm^{2}_-(\C\mathbb{P}^2)\cong \mathfrak{su}(2)\subset \mathfrak{su}(3)\cong \Lm^{1,1}_0(\mathcal{O}_{\C\mathbb{P}^2}(-3))$. Note however that the Levi-Civita connection of $\C\mathbb{P}^2$ does {not} pullback to a Hermitian Yang-Mills connection on $\mathcal{O}_{\C\mathbb{P}^2}(-3).$

Next, equipped with the explicit expression for $\Om$ given by Proposition \ref{holomorphicvolumeformCP2}, we initiate the search for special Lagrangian submanifolds in $\mathcal{O}_{\C\mathbb{P}^2}(-3)$.
\subsection{Special Lagrangians in $\mathcal{O}_{\C\mathbb{P}^2}(-3)$. }

We begin with the observation that the distribution $S\subset T(\mathcal{O}_{\C\mathbb{P}^2}(-3))$ defined by the triple 
$$S:=\langle \partial_r, \partial_t, \cos(3y) \partial_{\sigma_2}- \sin(3y) \partial_{\sigma_3}\rangle$$ 
is involutive (i.e. it is closed under the Lie bracket) and hence this defines locally a foliation of $\mathcal{O}_{\C \mathbb{P}^2}(-3)$. It is easy to see from (\ref{omCY}) that this is in fact a Lagrangian foliation i.e. $\om_{K}\big|_S=0$. From Proposition \ref{holomorphicvolumeformCP2} we compute 
\begin{align*}
	\Om^-(\partial_r, \partial_t, \cos(3y) \partial_{\sigma_2}- \sin(3y) \partial_{\sigma_3}) = &\sin(3y)(\frac{r^5}{2(r^6-C)^{1/2}(1-t)^{1/2}})(\cos(3y))-\\
	&\cos(3y)(\frac{r^5}{2(r^6-C)^{1/2}(1-t)^{1/2}})(\sin(3y))\\
	=&\ 0
\end{align*}
and hence it follows that $\Om^+$ calibrates $S$. Thus, $S$ defines a special Lagrangian foliation. 
To understand the topology of this foliation we first rewrite the Calabi-Yau metric (\ref{calabiyaumetric}) as
\begin{align*}
	g_{K} =&\ \Big(\frac{r^6}{r^6-C}\Big) dr^2 + \Big(\frac{ r^6 -C}{r^4}\Big) \theta^2 + r^2 \Big(ds^2+\sin^2(s)\cos^2(s) \sigma_1^2\\ 
	&+\cos^2(s) \big( (\sin(3y)\sigma_2+\cos(3y)\sigma_3 )^2 +(\cos(3y)\sigma_2-\sin(3y)\sigma_3 )^2\big)\Big)
\end{align*}
and from this we immediately see that
\[g_{K}\Big|_S=\Big(\frac{r^6}{r^6-C}\Big) dr^2 + r^2 \big(ds^2+\cos^2(s) ( \cos(3y)\sigma_2-\sin(3y)\sigma_3 )^2\big).\]
If we fix the value of $r$ then one recognises easily that $g_{K}\Big|_S$ restricts to the round metric on this hypersurface, but since $s \in [0, \pi/2]$ this only corresponds to a hemisphere (contained in $\C\mathbb{P}^2$). Thus, if $C\neq 0$ then each special Lagrangian corresponds to a real half-line bundle over this hemisphere and the induced metric is asymptotically Euclidean.
On the other hand if $C=0$ (so that the total space is now $\R^6$) then each special Lagrangian corresponds to a cone on this hemisphere i.e it is isometric to $\R^+_0\times \R^2$. 

As described in Section \ref{calabiyauonC3section} recall that here we are viewing $\C \mathbb{P}^2$ as a cohomogeneity one manifold with principal orbit $SU(2)=S^3$ and the two singular orbits are a point (at $t=0$) and an $S^2$ (at $t=1$ where the $1$-form $\sigma_1$ collapses i.e. the associated circle fibres shrink to a point). For each fixed $y$, consider the $S^1$ fibres in the principal orbit $S^3$ defined instead by the vector field $\cos(3y) \partial_{\sigma_2}- \sin(3y) \partial_{\sigma_3} \in S$ (note that unlike $\partial_{\sigma_1}$ this is not a Killing vector field on $\C\mathbb{P}^2$). As $t \to 0$ these $S^1$ fibres all shrink down to the same point and on the other end as $t \to 1$ the $S^1$ fibres in the principal $S^3$ orbit descend to geodesics in the singular $S^2$ orbit so that any two $S^1$ fibres now intersect at exactly $2$ points in the $S^2$. So strictly speaking $S$ does not determine a foliation globally as the leaves all intersect. However once the two singular orbits are removed (which corresponds to removing the complex submanifolds $\mathcal{O}_{\C\mathbb{P}^2}(-3)\big|_{S^2}\cong \mathcal{O}_{\C\mathbb{P}^1}(-3)$ and $\mathcal{O}_{\C\mathbb{P}^2}(-3)\big|_{pt}\cong \C$ from $\mathcal{O}_{\C\mathbb{P}^2}(-3)$) then we indeed get a special Lagrangian foliation of the complement. This foliation is locally parametrised by $S^1 \times S^2$ (the $S^1$ coming from the $y$-coordinate and the $S^2\simeq S^3/S^1$ is parametrising the $S^1$ fibres of $\cos(3y) \partial_{\sigma_2}- \sin(3y) \partial_{\sigma_3}$ in principal $S^3$ orbit).
Thus, to summarise the above discussion we have shown:
\begin{Th}\label{speciallagfoliation}
The complement of the two complex submanifolds $\mathcal{O}_{\C\mathbb{P}^2}(-3)\big|_{S^2}$ and $\mathcal{O}_{\C\mathbb{P}^2}(-3)\big|_{pt}$
in the canonical bundle $\mathcal{O}_{\C\mathbb{P}^2}(-3)$ admits a foliation by special Lagrangians whose leaves topologically correspond to a real half-line bundle over a hemisphere. In the limiting case when $C \to 0$ so that $g_{K}\to g_{\C^3}$, the leaves converge to $\R^+_0 \times \R^2$ endowed with its flat metric to give a special Lagrangian foliation of $\C^3$.
\end{Th}
\begin{Rem}
Strictly speaking as $C\to 0$, $g_{K}\to g_{\C^3/\mathbb{Z}_3}$ since $y\in[0,2\pi/3)$ but we can pullback to the universal cover $\C^3$. 
Special Lagrangians on $\mathcal{O}_{\C\mathbb{P}^2}(-3)$ were also constructed by Goldstein in \cite{Goldstein2001} but by a different approaching using the moment map of a $\mathbb{T}^2$ action preserving the Calabi-Yau structure. Since this is not the case in our situation, our examples appear to be different to those found in \cite{Goldstein2001}.
\end{Rem}
Roughly, the SYZ version of mirror symmetry \cite{SYZ1996} asserts that, in certain suitable limits, mirror pairs of compact Calabi-Yau $3$-fold admit fibrations by approximately flat special Lagrangians $\mathbb{T}^3$ together with flat $U(1)$ connections, see \cite[Sect. 9.4]{Joycebook}. Although our examples are not strictly of these kinds they do share several similar features and as such they
might provide explicit local models. 
It is worth pointing out that the finite energy $U(1)$ connection form $\underline{A}$ given by (\ref{firstone}) restricts on a flat connection on the above special Lagrangians i.e. $F_{\underline{A}}\big|_S=0$.

\section{Abelian instantons on the canonical bundle of $S^2 \times S^2$}\label{S2S2instantonsection}
In this section we consider the analogous problem as in the last section but now for the canonical bundle of $S^2 \times S^2$, which we denote by $K_{1,-1}$.  Our goal here is to illustrate the similarities and differences between the two situations.
\subsection{The Calabi-Yau metric on $K_{1,-1}$}
To emphasise the similarity with the construction of the Calabi-Yau metric on the canonical bundle of $\C\mathbb{P}^2$, we begin by considering  
a product K\"ahler structure on $S^2_1 \times S^2_2=\C \mathbb{P}^1_1 \times \overline{\C \mathbb{P}^1_2}$ given by the pair
\begin{gather*}
\om_{S^2\times S^2}:=\om_{S^2_1}-\om_{S^2_2}= \frac{r_1}{(1+r^2_1)^2} dr_1 \w d\theta_1-\frac{r_2}{(1+r_2^2)^2} dr_2 \w d\theta_2,\\
g_{S^2 \times S^2}:=g_{S^2_1}+g_{S^2_2}=\frac{1}{(1+r_1^2)^2}(dr_1^2 + r_1^2 d\theta_1^2)+\frac{1}{(1+r_2^2)^2}(dr_2^2 + r_2^2 d\theta_2^2).
\end{gather*}
Note that here we considering the opposite orientation to the usual Fubini-Study one on $S^2_2$. In the above convention we have that $\Ric(g_{S^2 \times S^2})=4  g_{S^2 \times S^2}$.
On the $S^1$ bundle $Q^5\to S^2_1 \times S^2_2$ determined by the cohomology class 
$$[\om_{S^2\times S^2}]=[\om_{S^2_1}]-[\om_{S^2_2}]\in H^2(S^2_1 \times S^2_2,\mathbb{Z}) \cong H^2(S^2_1,\mathbb{Z}) \oplus H^2(S^2_2,\mathbb{Z}),$$ 
we define a connection $1$-form by
\begin{equation}
\eta:=dy-\frac{1}{2(1+r_1^2)}d\theta_1+\frac{1}{2(1+r_2^2)}d\theta_2,
\end{equation}
where $y$ again denotes the fibre coordinate, so that indeed $d\eta=\om_{S^2\times S^2}$.  Then the conical Calabi-Yau metric on $\R^+ \times Q^5$ is explicitly given by
\begin{gather*}
g_{c}:=dr^2 +r^2(\frac{16}{9}\eta^2+\frac{2}{3}g_{S^2 \times S^2}),\\
\om_{c}:= \frac{4}{3}r dr \w \eta + \frac{2}{3} r^2 \om_{S^2\times S^2}.
\end{gather*}
Analogously to the previous case we can remove the conical singularity by adding an $S^2 \times S^2$ at the apex; we denote this space by $K_{1,-1}$ endowed with the complete Calabi-Yau metric
\begin{gather*}
	g_{K_{1,-1}}:=\Big(\frac{r^6}{r^6-C}\Big)dr^2 +\frac{16}{9}\Big(\frac{r^6-C}{r^4}\Big)\eta^2+\frac{2}{3}r^2g_{S^2 \times S^2},
\end{gather*}
and as before the K\"ahler form is unchanged i.e. $\om_{K_{1,-1}}=\om_{c}$. We can define an obvious constant length $(3,0)$-form by
\[\psi^++i\psi^-= \frac{2r^2}{3(r_1^2+1)(r_2^2+1)}\Big(\frac{r^3 dr}{(r^6-C)^{1/2}}+i \frac{4(r^6-C)^{1/2}}{3r^2}\eta \Big)\w dz_1\w d\bar{z}_2.\]
where $z_j:=x_j+iy_j$, $x_j:=r_j \cos (\theta_j)$ and $y_j:=r_j \sin (\theta_j)$. Note again that the conjugate sign on $z_2$ reflects the fact that we are considering the negative of the standard Fubini-Study complex structure on $S^2_2$. One can then show that:
\begin{Prop}\label{holomorphicvolumeformS2S2}
	The pair
	\begin{gather}
		\Om^+:=+\sin(4y+2\theta_1-2\theta_2)\psi^+ + \cos(4y+2\theta_1-2\theta_2)\psi^-,\\
		\Om^-:=-\cos(4y+2\theta_1-2\theta_2)\psi^+ + \sin(4y+2\theta_1-2\theta_2) \psi^-,
	\end{gather}
	defines a holomorphic volume form i.e. $d\Om^\pm=0$ and satisfies (\ref{compatibilitysu3}) and (\ref{normalisationsu3}).
\end{Prop}
\begin{Rem}
	In contrast to Proposition \ref{holomorphicvolumeformCP2} observe that now the rotating function depends on both $M^4=S^2\times S^2$ and the fibre coordinate $y$. This makes calculations a bit more involved but it is nonetheless straightforward to verify that $d\Om^{\pm}=0.$ 
\end{Rem}
Unfortunately in this case we have not been able to find any special Lagrangian fibration. There are however several examples of special Lagrangians: 

\textit{Examples of special Lagrangians.}
Consider the foliation defined by the distribution $\langle \partial_{r_1},\partial_{r_2},\partial_{r} \rangle$. This foliation is of course (locally) parametrised by $(y,\theta_1,\theta_2)$ and it is easy to see that this indeed defines a Lagrangian foliation.
One can verify that for $\theta_1=\theta_2$ and $y=\frac{\pi}{8}$ and for $\theta_1=\theta_2+\frac{\pi}{2}$ and $y=0$ the corresponding leaves are in fact special Lagrangians; this gives two $1$-parameter families. Again we observe that as $C\to 0 $ the special Lagrangians converge to flat cones but now over $S^1 \times S^1 \subset S^2 \times S^2$. 

\subsection{Instantons: Abelian examples}
The computations of sub-section \ref{S1invariantInstantonAntiCanonical}
carry over immediately to the present situation by replacing $\om_{FS}$ by $\frac{2}{3}\om_{S^2\times S^2}$ and $\theta$ by $\frac{4}{3}\eta$. So we are again required to solve (\ref{fundamentalode}) but now instead of $(\C\mathbb{P}^2,g_{FS})$ we have $(S^2 \times S^2, \frac{2}{3}g_{S^2\times S^2})$.
From \cite[Chap. 3 Corollary C. I. 3]{Berger1971spectre} we easily deduce that the eigenvalues of the latter are given by 
\begin{equation}
\mu_k=6 k(k+1)
\end{equation} 
for $k\in\mathbb{Z}^+_0$ and $\mu_k$ occur with multiplicity 
\begin{equation}
m_k=2(2k+1).
\end{equation}
Here the eigenfunctions are simply linear combinations of spherical harmonics, with same eigenvalue, on each $S^2$ factor.
Solving for $\kappa(r)$ as in the previous section we find that 
\begin{align*}
	\kappa(r) =\ &c_1  (r^6-C)\ _2F_1\Big(\frac{8+\sqrt{6k^2+6k+4}}{6},\frac{8-\sqrt{6k^2+6k+4}}{6};\frac{5}{3};\frac{r^6}{C}\Big)\ +\\ 
	&c_2 \frac{ (r^6-C)}{r^4}\ _2F_1\Big(\frac{4+\sqrt{6k^2+6k+4}}{6},\frac{4-\sqrt{6k^2+6k+4}}{6};\frac{1}{3};\frac{r^6}{C}\Big).
\end{align*}
As before we know that we get a well-defined solution for $r\in[C^{1/6},+\infty)$ if either of the first two entry of $_2F_1$ is a non-positive integer; here are a few values of $k$ we found using $\mathsf{Python}$ for which this occurs: $k=0,15,64,1519,6320,148895,619344$ and $k=1,6,153,638,15041,62566$.

The case when $k=0$ is the same for all $M^4$, so we always have the instanton (\ref{firstone}) with finite Yang-Mills energy essentially corresponding to the volume of the zero section $M^4\hookrightarrow P^6$.
When $k=1$ i.e. $\mu=12$ we get the same solution $\kappa$ as before but the multiplicities of the eigenvalues differ: $m_1(S^2\times S^2)=6$ whereas $m_1(\C\mathbb{P}^2)=32$.

When $k=6$ i.e. $\mu=252$ we get
\[\kappa(r)=\frac{(r^6-C)(65 r^{12} + 40 C r^6 + 2C^2)}{r^4}.\]
Now $m_6(S^2\times S^2)=26$ whereas $m_7(\C\mathbb{P}^2)=2048$ ($\mu_7(\C\mathbb{P}^2)=252$).
When $k=15$ i.e. $\mu=1440$ we get
\[\kappa(r)=(r^6 + C)(C^5 + 23C^4r^6 + \frac{299}{2}C^3r^{12} + \frac{8671}{22}C^2r^{18} + \frac{34684}{77}Cr^{24} + \frac{34684}{187}r^{30}).\]
Now $m_{15}(S^2\times S^2)=62$ whereas $m_{18}(\C\mathbb{P}^2)=27436$ ($\mu_{18}(\C\mathbb{P}^2)=1440$).

A general conclusion of the above is that there are considerably fewer abelian instantons on $K_{1,-1}$ than on $\mathcal{O}_{\C\mathbb{P}^2}(-3)$. 

\section{Abelian instantons on other K\"ahler Einstein $3$-folds}\label{T4section}
Having studied in detail abelian instantons on  Calabi-Yau $3$-folds arising from the Calabi ansatz in case (\ref{calabiansatzgenral}), we shall now consider instantons on the K\"ahler Einstein manifolds arising from cases (\ref{CP3asS1invariant})-(\ref{EinsteinonHK}) and also on the half-complete examples arising from (\ref{uhalfcomplete}). The goal here is to investigate how the set of abelian instantons on $P^6$ depends on the K\"ahler Einstein manifold $M^4$.

\subsubsection{Case \ref{CP3asS1invariant}}\label{instantononCP3}
In this case we have that $M^4$ is a positive K\"ahler Einstein $4$-manifold and $u$ is given by
\begin{equation}
	u=\frac{1}{2H^{1/2}(1-H)^{1/2}},
\end{equation}
with $H\in [0,1].$ Recall that, aside from the case when $M^4=\C\mathbb{P}^2$, $P^6$ always has a conical singularity at $H=0$.
Solving (\ref{separableabeliancase}) for $\kappa$ we find that
\begin{align*}
	\kappa =\ &c_1 H^{-1-\frac{\sqrt{4+\mu}}{2}}(1-H)\ _2F_1\Big(-\frac{\sqrt{4+\mu}}{2},2-\frac{\sqrt{4+\mu}}{2};1-\sqrt{4+\mu};H\Big)\\
	&+ c_2 H^{-1+\frac{\sqrt{4+\mu}}{2}}(1-H)\ _2F_1\Big(\frac{\sqrt{4+\mu}}{2},2+\frac{\sqrt{4+\mu}}{2};1+\sqrt{4+\mu};H\Big).
\end{align*}
Since $M^4$ is necessarily compact we know that $\mu \geq 0$ and hence it is not hard to see that any solution $\kappa$ always has a singularity; for instance when $M^4=\C\mathbb{P}^2$ so that $\mu=4k(k+2)$ we have that
\begin{align*}
	\kappa =\ &c_1 H^{-k-2}(1-H)\ _2F_1\Big(-1-k,1-k;-1-2k;H\Big)\\
	&+ c_2 H^{k}(1-H)\ _2F_1\Big(1+k,3+k;3+2k;H\Big).
\end{align*}
Indeed we see that we always have a singularity at $H=0$, as we already knew from Remark \ref{CP3remark}. 
\subsubsection{Case \ref{MoreEinstein}}
In this case we have that $M^4$ is a negative K\"ahler Einstein $4$-manifold and $u$ is given by
\begin{equation}
	u=\frac{1}{2H^{1/2}(H-1)^{1/2}},
\end{equation}
with $H\in [1,\infty).$
Solving (\ref{separableabeliancase}) for $\kappa$ we find that
\begin{align*}
	\kappa =\ &c_1 H^{-1-\frac{\sqrt{4-\mu}}{2}}(H-1)\ _2F_1\Big(-\frac{\sqrt{4-\mu}}{2},2-\frac{\sqrt{4-\mu}}{2};1-\sqrt{4-\mu};H\Big)\\
	&+ c_2 H^{-1+\frac{\sqrt{4-\mu}}{2}}(H-1)\ _2F_1\Big(\frac{\sqrt{4-\mu}}{2},2+\frac{\sqrt{4-\mu}}{2};1+\sqrt{4-\mu};H\Big).
\end{align*}
By contrast to the previous case since $M^4$ can now be non-compact one can also have negative values for $\mu$: for instance $\mu=-12$ gives
$\kappa ={H^{-3}(H-1)}$ which is gobally well-defined since $H\geq 1$. 
A rather interesting observation is that in this case the only positive value $\mu$ possible is $4$: we do not know if this actually occurs.
\subsubsection{Case \ref{EinsteinonHK}}
Let $(M^4,g_M,\om_1,\om_2,\om_3)$ be a hyperK\"ahler manifold such that $[\om_1]\in H^2(M^4,\mathbb{Z})$ and let $P^6=\R_t \times Q^5$, where $Q^5$ denotes the total space of the $S^1$ bundle determined by $[\om_1]$. From (\ref{EinsteinonHK}), we can define a K\"ahler Einstein structure on $P^6$ by
\begin{gather*}
	g=dt^2 + 4 e^{4t}\Theta^2 + e^{2t}g_M,\\
	\om=2e^{2t}\Theta \w dt+e^{2t}\om_1,
\end{gather*}
where in the previous notation we set $a=1$, $\lambda=16$ and $H=e^{2t}$.
Solving (\ref{separableabeliancase}) for $\kappa(t)$ we find that 
\[\kappa(t)=c_1  e^{-2t}I_2(\mu^{1/2}e^{-t})+c_2 e^{-2t}K_2(\mu^{1/2}e^{-t}),\]
where $I_2$ and $K_2$ denote the modified Bessel functions. By contrast to the situation when $P^6$ was a Calabi-Yau manifold, now we see that the solution $\kappa$ is well-defined for all eigenvalues $\mu$.
If $M^4$ is compact, then from Kodaira's classification of surfaces we know that $M^4$ is either $\mathbb{T}^4$ or a $K3$ surface cf. \cite{Barth}. When $M^4=\mathbb{T}^4=\R^4/\Gamma$, the universal cover of $P^6$ in fact corresponds to the complex hyperbolic space with constant holomorphic sectional curvature $-4$ and the eigenfunctions of the Laplacian are generated by $\sin(k_i x_i)$ and $\cos(k_i x_i)$, where ${k}=(k_1,k_2,k_3,k_4)$ are vectors depending on the co-compact lattice $\Gamma$ and satisfy $\sum k_i^2=\mu$. So we see that the moduli space of instantons already distinguishes between $4$-tori with different lattices.
For the case when $M^4$ is a $K3$ surface, we are unaware of any results pertaining to the eigenvalues of the Laplacian. 

If $M^4$ is non-compact then $\Delta_M$ can also have negative eigenvalues $\mu$ and hence $\kappa(t)$ will also consist of the standard Bessel functions $J_2$ and $Y_2$. The simplest example is when $M^4=\R^4$ in which case we can take $\exp(ax)$ as eigenfunctions.
More generally we can take \textit{suitable} $M^4$ as arising from the Gibbons-Hawking ansatz, see \cite[Sect. 4.2]{UdhavFowdar4}, and hence construct instantons on $P^6$ with $M^4$ as  muti-Eguchi-Hanson and multi-Taub-NUT manifolds.

\subsubsection{Conti-Salamon examples}
We now consider the half complete example which arises from (\ref{uhalfcomplete}) when $a=1$ and $b=p=q=0$ so that $u=H>0$. Note that in this case $M^4$ can be any hyperK\"ahler manifold with $[\om_1]\in H^2(M^4,\mathbb{Z})$ as above. Solving (\ref{separableabeliancase}) for $\kappa$ we now find that 
\[\kappa=\frac{c_1}{H}J_{2/3}\Big(\frac{2H^{3/2}\sqrt{-\mu}}{3}\Big)+\frac{c_2}{H}Y_{2/3}\Big(\frac{2H^{3/2}\sqrt{-\mu}}{3}\Big).\]
We see that the corresponding instanton has a singularity at $H=0$ (as does the metric). Observe that as above if we take $M^4$ to be compact then $\mu \geq 0$ while if $M^4$ is non-compact then we can also have negative values of $\mu$.

Thus, the results of this section illustrate how the construction of instantons on the $S^1$-invariant K\"ahler Einstein manifold $P^6$ depends on $M^4$ and $u$. In future work we hope to study equation (\ref{parabolicequforpsi}) when the gauge group is non-abelian.

\section{Examples of deformed Hermitian Yang-Mills connections}\label{dhymsection}

In this section we construct explicit examples of $S^1$-invariant deformed Hermitian Yang-Mills connections on $P^6$ as given by (\ref{calabiansatzgenral})-(\ref{EinsteinonHK}).

We consider the set up of Theorem \ref{S1invariantdhymtheorem} when $\psi=\kappa(H)$ is only a function of $H$ i.e. $d_M\psi=0$, and $A=0.$ From (\ref{dhymreduced}) one readily sees that $\kappa(H)$ then satisfies:
\begin{equation}
	\kappa'(H)(\kappa(H)^2-H^2)-2H \kappa(H)=0.\label{dhymsymmetry}
\end{equation}

Note that although the gauge group is abelian the deformed Hermitian Yang-Mills equation (\ref{dhym}) constitute a non-linear PDE in general, which owing to our ansatz is reduced to an ODE in this case. 
Observe also that (\ref{dhymsymmetry}) is independent of the function $u$ (as given by  (\ref{solutionforuEinstein})), which is analogous to the result in Section \ref{CP2section} in the case when the eigenvalue $\mu=0$.

Using the substitution $\kappa(H)=Hz(H)$ in (\ref{dhymsymmetry}) we get a separable equation for $z(H)$ and solving the resulting ODE we find that the general solution to (\ref{dhymsymmetry}) is given by the cubic equation:
\begin{equation}
\kappa^3-3H^2\kappa -c/4=0,\label{cubicequation}
\end{equation} 
where $c$ is a constant. When $c=0$ we get three linear solutions to (\ref{cubicequation}) which intersect at the origin while when $c \neq 0$ we obtain six solutions, each asymptotic the latter linear ones (see Figure 1).  Recall that if ${\underline{A}}$ is a deformed Hermitian Yang-Mills connection then so is $-{\underline{A}}$. Indeed this $\mathbb{Z}_2$-symmetry is reflected in Figure 1 by reflection in the $H$-axis, so without loss of generality we shall restrict to the case when $c \geq 0$. Observe that there is also an additional $\mathbb{Z}_3$-symmetry in Figure 1 relating the latter solutions.

\begin{figure}
	\centering
	\includegraphics[height=7cm]{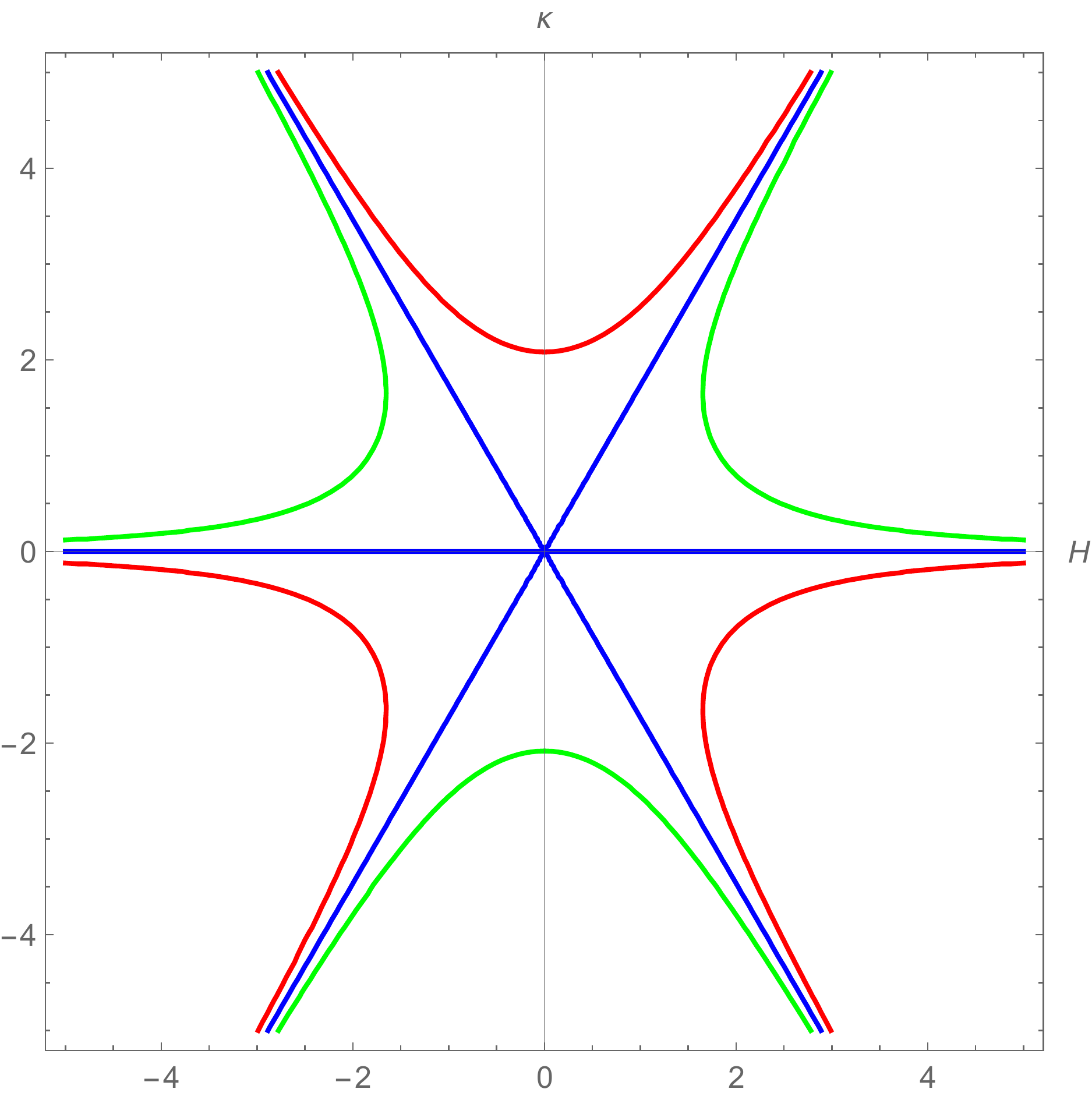}
	\caption{{Plots of (\ref{cubicequation}) for \color{red}{$c>0$}\color{black}, \color{blue}{$c=0$}\color{black}, \color{green}{$c<0$}\color{black}{.}}}
\end{figure}

When $c=0$, from (\ref{cubicequation}) we have that either $\kappa=0$ which corresponds to the flat connection, or $\kappa(H)=\pm H\sqrt{3}$ and hence $F_{\underline{A}}=\mp\sqrt{3}\!\ \om$: this recovers Example 3.4 described in \cite{Lotay2020}. So the solutions to (\ref{cubicequation}) can be considered as deformation of these examples. It is worth emphasising that when $c=0$ the above solutions are in fact Hermitian Yang-Mills connections while for $c> 0$ the solutions are only deformed Hermitian Yang-Mills connections. 
In particular, we see from Figure 1 that there is always one solution well-defined for all $H\in \R$ and thus, we have that:
\begin{Th}\label{dhymtheorem}
On any of the K\"ahler Einstein manifold given in (\ref{calabiansatzgenral})-(\ref{EinsteinonHK}), there exists a $1$-parameter family of deformed Hermitian Yang-Mills connection of the form $\underline{A}=\kappa(H)\Theta$, where $\kappa$ is determined by $c> 0$ via (\ref{cubicequation}). Moreover, as $c \to 0$ $\underline{A}$ converges to the Hermitian Yang-Mills connection $H\sqrt{3}\Theta$ (up to a $\mathbb{Z}_2$ ambiguity). 
\end{Th}


We now consider the aforementioned $\mathbb{Z}_3$-symmetry.
In the case when $P^6$ is as given by the Calabi ansatz, we know that the Calabi-Yau metric (\ref{calabiyaumetric}) is well-defined for $r\in [C^{1/6},\infty)$, where recall that $r^2=2H$, and $C$ essentially corresponds to the size of the zero section. 
From Figure 1 we now see that if $c$ is sufficiently large compared to $C$ then we only get one $1$-parameter family of solutions well-defined for $r\in [C^{1/6},\infty)$.
On the other hand if $c$ is small enough compared to $C$ then we get $3$ distinct $1$-parameter families of solutions: two of which are asymptotic to the solutions $\kappa=\pm H\sqrt{3}$ and one asymptotic to $\kappa=0$. 
Thus, we have that:
\begin{Th}\label{dhymC3theorem}
 On $\C^3$ there exists one $1$-parameter family of deformed Hermitian Yang-Mills connections of the form $\underline{A}=\kappa(H)\Theta$, as given by Theorem \ref{dhymtheorem}, whereas on $\mathcal{O}_{\C\mathbb{P}^2}(-3)$ there exist three such $1$-parameter families. 
\end{Th}
In particular, the preceding argument implies that while the flat connection can be deformed to a deformed Hermitian Yang-Mills connection on $\mathcal{O}_{\C\mathbb{P}^2}(-3)$, this cannot be done on $\C^3$, at least not via a family of the form $\underline{A}=\kappa(H)\Theta$.

\bibliography{biblioG}
\bibliographystyle{plain}

\end{document}